\documentclass[reqno,11pt]{amsart}
\usepackage{amssymb}
\usepackage{amsmath}
\usepackage{stmaryrd}
\usepackage{graphicx}
\usepackage{color}
\usepackage{cite}
\usepackage[font=small]{caption}
\usepackage{tikz}
\usepackage{pgfplots}
\pgfplotsset{compat=1.10}
\usepgfplotslibrary{fillbetween}
\usetikzlibrary{patterns}
\newtheorem{theorem}{Theorem}[section]

\newtheorem{lemma}[theorem]{Lemma}
\newtheorem{proposition}[theorem]{Proposition}

\newtheorem{remark}[theorem]{Remark}
\newtheorem{example}[theorem]{Example}
\newcommand{\R}{\mathbb{R}}

\newcommand{\rd}{\mathrm{d}}
\definecolor{cadmiumgreen}{rgb}{0.0, 0.42, 0.24}
\numberwithin{equation}{section}
\begin{document}
%%%%%%%%%%%%%%%%
%%%%%%%%%%%%%%%%

%%%%%%%%%%%%%%%%
\title[$H^2$-regularity for a 2d transmission problem with geometric constraint]{$H^2$-regularity for a two-dimensional transmission problem with geometric constraint}
%%%%%%%%%%%%%%%%

%%%%%%%%%%%%%%%%
\author{Philippe Lauren\c{c}ot}
\address{Institut de Math\'ematiques de Toulouse, UMR~5219, Universit\'e de Toulouse, CNRS \\ F--31062 Toulouse Cedex 9, France}
\email{laurenco@math.univ-toulouse.fr}
\author{Christoph Walker}
\address{Leibniz Universit\"at Hannover\\ Institut f\" ur Angewandte Mathematik \\ Welfengarten 1 \\ D--30167 Hannover\\ Germany}
\email{walker@ifam.uni-hannover.de}
%%%%%%%%%%%%%%%%
%
%\thanks{**}
%
\date{\today}
%
%
%%%%%%%%%%%%%%%%
%%%%%%%%%%%%%%%%
\begin{abstract}
The $H^2$-regularity of variational solutions to a two-dimensional transmission problem with geometric constraint is investigated, in particular when part of the interface becomes part of the outer boundary of the domain due to the saturation of the  geometric constraint. In such a situation, the domain includes some non-Lipschitz subdomains with cusp points, but it is shown that this feature does not lead to a regularity breakdown.  Moreover, continuous dependence of the solutions with respect to the domain is established.
\end{abstract}
%%%%%%%%%%%%%%%%
%%%%%%%%%%%%%%%%
%
\maketitle
%
%%%%%%%%%%%%%%%%
%%%%%%%%%%%%%%%%
\section{Introduction}\label{IMR}
%%%%%%%%%%%%%%%%
%%%%%%%%%%%%%%%%

The $H^2$-regularity of variational solutions to a two-dimensional transmission problem with geometric constraint is investigated, in particular when part of the interface becomes part of the outer boundary of the domain due to the geometric constraint, a situation in which the domain includes some non-Lipschitz subdomains with cusp points. 

To set up the geometric framework, let $D:=(-L,L)$ be a finite interval of $\mathbb{R}$, $L>0$, and let $H>0$ and $d>0$ be two positive parameters. Given a function $u\in C(\bar{D},[-H,\infty))$ with $u(\pm L)=0$, we define the subdomain $\Omega(u)$ of $D\times (-H,\infty)$ by
$$
 \Omega(u) := \left\{ (x,z)\in D\times \mathbb{R} \,:\, -H<  z <  u(x)+d \right\} = \Omega_1(u)\cup \Omega_2(u)\cup \Sigma(u)\,,
$$
where 
$$
\Omega_1(u) := \left\{ (x, z)\in D\times \mathbb{R} \,:\, -H< z< {u}(x)\right\}
$$
and 
$$
\Omega_2(u):= \left\{ (x,z)\in D\times \mathbb{R}\,:\, u(x)<  z <  u(x)+d\right\}
$$
are separated by the interface
$$
\Sigma(u) := \left\{ (x,z)\in D\times \mathbb{R}\,:\, z=  u(x)>-H \right\}\,.
$$
%%%%%%%%%%%%%%%%
%%%%%%%%%%%%%%%%
\begin{figure}
	\begin{tikzpicture}[scale=0.9]
		\draw[black, line width = 1.5pt, dashed] (-7,0)--(7,0);
		\draw[black, line width = 2pt] (-7,0)--(-7,-2.5);
		\draw[black, line width = 2pt] (7,-2.5)--(7,0);
		\draw[black, line width = 2pt] (-7,-2.5)--(7,-2.5);
		\draw[blue, line width = 2pt] plot[domain=-7:7] (\x,{-0.75-0.75*cos((pi*\x/7) r)});
		\draw[blue, line width = 2pt] plot[domain=-7:7] (\x,{-1.25-0.75*cos((pi*\x/7) r)});
		\draw[blue, line width = 1pt, arrows=->] (2,0)--(2,-1.65);
		\node at (2.2,-0.6) {${\color{blue} v}$};
		\node at (-5,-1.5) {${\color{blue} \Omega_1(v)}$};
		\node at (-3,0.5) {${\color{blue} \Omega_2(v)}$};
		\draw (-3.6,0.5) edge[->,bend right,line width = 1pt] (-4.7,-0.7);
		\node at (0,-3.25) {$D$};
		\node at (5.75,-1.75) {{\color{blue} $\Sigma(v)$}};
		\draw (5.25,-1.75) edge[->,bend left, line width = 1pt] (3.5,-1.3);
		\node at (-7.8,1) {$z$};
		\draw[black, line width = 1pt, arrows = ->] (-7.5,-3)--(-7.5,1);
		\node at (-8,-2.5) {$-H$};
		\draw[black, line width = 1pt] (-7.6,-2.5)--(-7.4,-2.5);
		\node at (-7.8,-0.5) {$0$};
		\draw[black, line width = 1pt] (-7.6,-0.5)--(-7.4,-0.5);
		\node at (-7.8,0) {$d$};
		\draw[black, line width = 1pt] (-7.6,0)--(-7.4,0);
		\node at (-7,-3.25) {$-L$};
		\node at (7,-3.25) {$L$};
		\draw[black, line width = 1pt, arrows = <->] (-7,-2.75)--(7,-2.75);
	\end{tikzpicture}
	\caption{Geometry of $\Omega(v)$ for a state $v\in S$ with empty coincidence set.}\label{Fig1}
		\begin{tikzpicture}[scale=0.9]
		\draw[black, line width = 1.5pt, dashed] (-7,0)--(7,0);
		\draw[black, line width = 2pt] (-7,0)--(-7,-2.5);
		\draw[black, line width = 2pt] (7,-2.5)--(7,0);
		\draw[black, line width = 2pt] (-7,-2.5)--(7,-2.5);
		\draw[red, line width = 2pt] plot[domain=-7:7] (\x,{-1-cos((pi*\x/7) r)});
		\draw[red, line width = 2pt] plot[domain=-7:7] (\x,{-1.5-cos((pi*\x/7) r)});
		\draw[red, line width = 1pt, arrows=->] (2,0)--(2,-2.05);
		\node at (2.3,-0.6) {${\color{red} w}$};
		\node at (-5,-1.5) {${\color{red} \Omega_1(w)}$};
		\node at (-3,0.5) {${\color{red} \Omega_2(w)}$};
		\draw (-3.65,0.5) edge[->,bend right,line width = 1pt] (-4.7,-0.75);
		\node at (0,-3.25) {$D$};
		\node at (5.75,-1.75) {{\color{red} $\Sigma(w)$}};
		\draw (5.25,-1.75) edge[->,bend left, line width = 1pt] (3.5,-1.6);
		\node at (-7.8,1) {$z$};
		\draw[black, line width = 1pt, arrows = ->] (-7.5,-3)--(-7.5,1);
		\node at (-8,-2.5) {$-H$};
		\draw[black, line width = 1pt] (-7.6,-2.5)--(-7.4,-2.5);
		\node at (-7.8,-0.5) {$0$};
		\draw[black, line width = 1pt] (-7.6,-0.5)--(-7.4,-0.5);
		\node at (-7.8,0) {$d$};
		\draw[black, line width = 1pt] (-7.6,0)--(-7.4,0);
		\node at (-7,-3.25) {$-L$};
		\node at (7,-3.25) {$L$};
		\draw[black, line width = 1pt, arrows = <->] (-7,-2.75)--(7,-2.75);
		\node at (-1.5,-3.25) {${\color{red} \mathcal{C}(w)}$};
		\draw (-1,-3.25) edge[->,bend right, line width = 1pt] (0,-2.55);
		\draw[black, line width = 2pt] (-7,-2.5)--(7,-2.5);
	\end{tikzpicture}
\caption{Geometry of $\Omega(w)$ for a state $w\in \mathcal{S}$ with non-empty coincidence set.}\label{Fig2}
\end{figure}
%%%%%%%%%%%%%%%%
%%%%%%%%%%%%%%%%
Owing to the (geometric) constraint $u\ge -H$, the lower boundary of $\Omega_2(u)$, given by the graph of the function $u$, cannot go beyond the lower boundary $D\times \{-H\}$ of $\Omega_1(u)$ but may coincide partly with it, along the so-called coincidence set 
\begin{equation}
	\mathcal{C}(u) := \{x\in D\,:\, u(x)=-H\}\,, \label{CS}
\end{equation}
see Figures~\ref{Fig1} and~\ref{Fig2}. Clearly, the geometry of $\Omega(u)$, as well as the regularity of its boundary, heavily depends on whether $\min_D\{u\}>-H$ or $\min_D\{u\}=-H$. Indeed, if $\min_D\{u\}>-H$ (i.e. the graph of $u$ is strictly separated from $D\times \{-H\}$ as in Figure~\ref{Fig1}), then the coincidence set $\mathcal{C}(u)$
is empty and $\Omega_1(u)$ is connected. In contrast, if $\min_D\{u\}=-H$ so that the graph of $u$ intersects $D\times \{-H\}$, then $\mathcal{C}(u)\ne\emptyset$ and $\Omega_1(u)$ is disconnected with at least two (and possibly infinitely many) connected components, see Figures~\ref{Fig2} and~\ref{Fig3}.

For such a geometry, we study the regularity of variational solutions to the transmission problem
\begin{subequations}\label{psi}
	\begin{align}
		\mathrm{div}(\sigma\nabla\psi_u) &=0 \quad\text{in }\ \Omega(u)\,,\label{a1a}\\
		\llbracket \psi_u \rrbracket =\llbracket \sigma\nabla \psi_u \rrbracket \cdot \mathbf{n}_{ \Sigma(u)} &=0 \quad\text{on }\ \Sigma(u)\,,\label{a1b}\\
		\psi_u&=h_u\quad\text{on }\ \partial\Omega(u)\,,\label{a1c}
	\end{align}
\end{subequations}
where 
\begin{equation*}
	\sigma := \sigma_1 \mathbf{1}_{\Omega_1(u)} + \sigma_2 \mathbf{1}_{\Omega_2(u)}
\end{equation*}
for some positive constants $\sigma_1\ne \sigma_2$, and $\mathbf{n}_{ \Sigma(u)}$ denotes the unit normal vector field to $\Sigma(u)$ (pointing into ${\Omega}_2(u)$) given by
$$
\mathbf{n}_{ \Sigma(u)}:=\frac{(-\partial_x u, 1)}{\sqrt{1+(\partial_x u)^2}}\,.
$$
In \eqref{a1c}, $h_u$ is a suitable function reflecting the boundary behavior of $\psi_u$, see Section~\ref{sec.abv} for details. In addition, $\llbracket \cdot \rrbracket$ denotes the (possible) jump across the interface $\Sigma(u)$; that is, 
\begin{equation*}
	\llbracket f \rrbracket(x,u(x)) := f|_{\Omega_1(u)}(x,u(x)) - f|_{\Omega_2(u)}(x,u(x))\,, \qquad x\in D\,,
\end{equation*}
whenever meaningful for a function $f:\Omega(u)\to\mathbb{R}$. 

\medskip

Let us already mention that there are several features of the specific geometry of $\Omega(u)$ which may hinder the $H^2$-regularity of the solution $\psi_u$ to \eqref{psi}. Indeed, on the one hand, the interface $\Sigma(u)$ always intersects with the boundary $\partial\Omega(u)$ of $\Omega(u)$ and it follows from \cite{Lem77} that this sole property prevents the $H^2$-regularity of $\psi_u$, unless $\sigma$ and the angles between $\Sigma(u)$ and $\partial\Omega(u)$ at the intersection points satisfy some additional conditions. On the other hand, $\Omega(u)$ and $\Omega_2(u)$ are at best Lipschitz domains, while $\Omega_1(u)$ may consist of non-Lipschitz domains with cusp points.

\medskip

The particular geometry $\Omega(u)= \Omega_1(u)\cup \Omega_2(u)\cup \Sigma(u)$, in which the boundary value problem~\eqref{psi} is set, is encountered in the investigation of an idealized electrostatically actuated microelectromechanical system (MEMS) as described in detail in~\cite{LW18}. Such a device consists of an elastic plate of thickness~$d$ which is fixed at its boundary $\{\pm L\} \times (0,d)$ and suspended above a rigid conducting ground plate located at $z=-H$. The elastic plate  is made up of a dielectric material and deformed by a Coulomb force induced by holding the ground plate and the top of the elastic plate at different electrostatic potentials. In this context, $u$ represents the vertical deflection of the bottom of the elastic plate, so that the elastic plate is given by $\Omega_2(u)$, while $\Omega_1(u)$ denotes the free space between the elastic plate and the ground plate. An important feature of the model is that the elastic plate cannot penetrate the ground plate, resulting on the geometric constraint $u\ge -H$. Still, a contact between the elastic plate and the ground plate --~corresponding to a non-empty coincidence set $\mathcal{C}(u)$~--  is explicitly allowed. The dielectric properties of $\Omega_1(u)$ and $\Omega_2(u)$ are characterized by positive constants $\sigma_1$ and $\sigma_2$, respectively. The electrostatic potential $\psi_u$ is then supposed to satisfy \eqref{psi} and is completely determined by the deflection $u$. The state of the MEMS device is thus described by the deflection $u$, and equilibrium configurations of the device are obtained as critical points of the total energy  which is the sum of the mechanical and electrostatic energies, the former being a functional of $u$ while the latter is the Dirichlet integral of $\psi_u$. Owing to the nonlocal dependence of $\psi_u$ on $u$, minimizing the total energy and deriving the associated Euler-Lagrange equation demand quite precise information on the regularity of the electrostatic potential $\psi_u$ for an arbitrary, but {\it fixed} function $u$ and its continuous dependence thereon. This first step of provisioning the required information is the main purpose of the present research, and we refer to the forthcoming paper~\cite{LW2?} where the minimizing problem leading to the determination of~$u$ is analyzed.

\medskip

Since the regularity of the variational solution $\psi_u$ to \eqref{psi} is intimately connected with the regularity of the boundaries of $\Omega(u)$, $\Omega_1(u)$, and $\Omega_2(u)$, let us first mention that $\Omega(u)$ and $\Omega_2(u)$ are always Lipschitz domains and that the measures of the angles at their vertices do not exceed $\pi$, a feature which complies with the $H^2$-regularity of $\psi_u$ away from the interface $\Sigma(u)$ \cite{Gr85}. This property is shared by $\Omega_1(u)$ when the coincidence set $\mathcal{C}(u)$ is empty, see Figure~\ref{Fig1}, so that it is expected that $\psi|_{\Omega_i(u)}$ belongs to $H^2(\Omega_i(u))$, $i=1,2$, in that case. However, when $\mathcal{C}(u)$ is non-empty, the open set $\Omega_1(u)$ is no longer connected and the boundary of its connected components is no longer Lipschitz, but features cusp points. Moreover, there is an interplay between the transmission conditions~\eqref{a1b} and the  boundary condition~\eqref{a1c} when $\mathcal{C}(u)\ne\emptyset$. Whether $\psi|_{\Omega_i(u)}$ still belongs to $H^2(\Omega_i(u))$, $i=1,2$, in this situation is thus an interesting question, that we answer positively in our first result. For the precise statement, we introduce the functional setting we shall work with in the sequel. Specifically, we set 
$$
\bar{\mathcal{S}} := \{v\in H^2(D) \cap H_0^1(D)\,:\, v\ge -H \text{ in } D \;\text{ and }\; \pm \llbracket \sigma \rrbracket \partial_x v(\pm L) \le 0 \}\,,
$$ 
and 
$$
\mathcal{S}:=\{v\in H^2(D) \cap H_0^1(D)\,:\, v> -H \text{ in } D  \;\text{ and }\; \pm \llbracket \sigma \rrbracket \partial_x v(\pm L) \le 0 \}\,.
$$ 
Clearly, the coincidence set $\mathcal{C}(u)$ is empty if and only if $u\in \mathcal{S}$. In addition, the situation already alluded to, where $\mathcal{C}(u)$ is non-empty and $\Omega_1(u)$ is a  disconnected open set in $\R^2$ with a non-Lipschitz boundary, corresponds to functions $u\in\bar{\mathcal{S}}\setminus \mathcal{S}$. Also, we include the constraint $\pm \llbracket \sigma \rrbracket \partial_x u(\pm L) \le 0$  in the definition of $\mathcal{S}$ and $\bar{\mathcal{S}}$ to guarantee that the way $\Sigma(u)$ and $\partial\Omega(u)$ intersect does not prevent the $H^2$-regularity of $\psi_u$ in smooth situations (i.e. $u\in \mathcal{S}\cap W_\infty^2(D)$), see \cite{Lem77}.

%%%%%%%%%%%%%%%%
\begin{theorem}\label{Thm1}
Suppose \eqref{bobbybrown} below. 
\begin{itemize}
\item[(a)] For each $u\in \bar{\mathcal{S}}$, there is a unique variational solution $\psi_u \in h_{u}+H_{0}^1(\Omega(u))$ to \eqref{psi}.  Moreover, $\psi_{u,1}:= \psi_{u}|_{\Omega_1(u)} \in H^2(\Omega_1(u))$ and $\psi_{u,2} := \psi_{u}|_{\Omega_2(u)} \in H^2(\Omega_2(u))$, and $\psi_{u}$  is a strong solution to the transmission problem~\eqref{psi}.
\item[(b)] Given $\kappa>0$, there is $c(\kappa)>0$ such that, for every $u\in\bar{\mathcal{S}}$ satisfying $\|u\|_{H^2(D)}\le \kappa$, 
\begin{equation*}
\|\psi_u\|_{H^1(\Omega(u))} + \|\psi_{u,1}\|_{H^2(\Omega_1(u))} + \|\psi_{u,2}\|_{H^2(\Omega_2(u))} \le c(\kappa)\,.
\end{equation*}	
\end{itemize}
\end{theorem}
%%%%%%%%%%%%%%%%

It is worth emphasizing that, for $i\in\{1,2\}$, the restriction of $\psi_u$ to $\Omega_i(u)$ belongs to $H^2(\Omega_i(u))$ for all $u\in\bar{\mathcal{S}}$. In particular, there is no regularity breakdown when the coincidence set $\mathcal{C}(u)$ is non-empty. A similar observation is made in \cite{ARMA20} for a different geometric setting when one of the two subsets does not depend on the function $u$. 

%%%%%%%%%%%%%%%
\begin{remark}\label{rem.clamped}
When the upper part $\Omega_2(v)$ is clamped at its lateral  boundaries in the sense that 
	\begin{equation*}
		u\in H_0^2(D) := \{v\in H^2(D)\cap H_0^1(D)\,:\, \partial_x v(\pm L) = 0\}\,,
	\end{equation*}
	Theorem~\ref{Thm1} applies whatever the values of $\sigma_1$ and $\sigma_2$. 
\end{remark}
%%%%%%%%%%%%%%%%

Theorem~\ref{Thm1} is an immediate consequence of Proposition~\ref{ACDC} below. Its proof begins with quantitative $H^2$-estimates on $\psi_u$ depending only on $\|u\|_{H^2(D)}$ for sufficiently smooth functions in $\mathcal{S}$, the $H^2$-regularity of $\psi_u$ being guaranteed by \cite{Lem77} in that case. Since the class of functions for which these estimates are valid is dense in $\bar{\mathcal{S}}$, we complete the proof with a compactness argument, the main difficulty to be faced being the dependence of $\Omega(u)$ on $u$. 

\noindent More precisely, we begin with a variational approach to \eqref{psi} and first show in Section~\ref{sec.vs} by classical arguments that, given $u\in \bar{\mathcal{S}}$,  the variational solution $\psi_{u}$ to \eqref{psi} corresponds to the minimizer on $h_{u}+H_{0}^1(\Omega(u))$ of the associated Dirichlet energy
\begin{equation*}
\mathcal{J}(u)[\theta]:= \frac{1}{2}\int_{\Omega(u)} \sigma |\nabla \theta|^2\,\rd (x,z)\,, \qquad \theta \in h_u + H_0^1(\Omega(u))\,. 
\end{equation*}
Thanks to this characterization, we use $\Gamma$-convergence tools to show the $H^1$-stability of $\psi_u$ with respect to $u$ in Section~\ref{sec.h1s}. Section~\ref{sec.h2r} is devoted to the study of the $H^2$-regularity of $\psi_u$ which we first establish in Section~\ref{sec.h2r1} for smooth functions $u\in \mathcal{S}\cap W_\infty^2(D)$ (thus having an empty coincidence set), relying on the analysis performed in \cite{Lem77}. It is worth mentioning that the constraint involving $\llbracket \sigma\rrbracket$ in the definition of $\mathcal{S}$ comes into play here. For $u\in \mathcal{S}\cap W_\infty^2(D)$, we next derive quantitative $H^2$-estimates on $\psi_u$ which only depend on $\|u\|_{H^2(D)}$ as stated in Theorem~\ref{Thm1}~(b), see Section~\ref{sec.h2e}. The building  block is an identity in the spirit of \cite[Lemma~4.3.1.2]{Gr85} allowing us to interchange derivatives with respect to $x$ and $z$ in some integrals involving second-order derivatives, its proof being provided in Appendix~\ref{sec.id}. We then combine these estimates with the already proved $H^1$-stability of variational solutions to \eqref{psi} and use a compactness argument to extend the $H^2$-regularity of $\psi_u$ to arbitrary functions $u\in \bar{\mathcal{S}}$ in Section~\ref{sec.h2e1}. In this step, special care is required to cope with the variation of the functional spaces with $u$. In fact, as a side product of the proof of Theorem~\ref{Thm1}, we obtain qualitative information on the continuous dependence of $\psi_u$ with respect to $u$, which we collect in the next result.

%%%%%%%%%%%%%%%%
\begin{theorem}\label{Thm2}
Suppose \eqref{bobbybrown} below. 
Let $\kappa>0$, $u\in \bar{\mathcal{S}}$, and consider a sequence $(u_n)_{n\ge 1}$ in $\bar{\mathcal{S}}$ such that 
\begin{equation}
	\|u_n\|_{H^2(D)}\le \kappa\,, \quad n\ge 1\,, \qquad \lim_{n\to\infty} \|u_n-u\|_{H^1(D)} = 0\,. \label{y0} 
\end{equation}
Setting $M := d + \max\left\{ \|u\|_{L_\infty(D)} \,,\, \sup_{n\ge 1}\{\|u_n\|_{L_\infty(D)}\} \right\}$, 
\begin{subequations}\label{y}
\begin{equation}
	\lim_{n\to \infty} \big\| (\psi_{u_n} - h_{u_n}) -(\psi_{u} - h_{u}) \big\|_{H^1(\Omega_M)} = 0\,. \label{y1}
\end{equation}
In addition, if $i\in\{1,2\}$ and $U_i$ is an open subset of $\Omega_i(u)$ such that $\bar{U}_i$ is a compact subset of $\Omega_i(u)$, then
\begin{equation}
	\psi_{u_n,i}\rightharpoonup \psi_{u,i} \quad\text{in}\quad H^2(U_i)\,. \label{y2}
\end{equation}
Also, for any $p\in [1,\infty)$, 
\begin{equation}
	\begin{split}
		\lim_{n\to\infty} \big\|\nabla\psi_{u_n,2}(\cdot,u_n) - \nabla\psi_{u,2}(\cdot,u) \big\|_{L_p(D,\R^2)} & = 0 \,, \\
		\lim_{n\to\infty} \big\| \nabla\psi_{u_n,2}(\cdot,u_n+d) - \nabla\psi_{u,2}(\cdot,u+d) \big\|_{L_p(D,\R^2)}  & = 0 \,.
	\end{split} \label{y3}
\end{equation}
\end{subequations}
\end{theorem}
%%%%%%%%%%%%%%%%

Clearly, the quantity $M$ introduced in Theorem~\ref{Thm2} is finite due to \eqref{y0} and the continuous embedding of $H^1(D)$ in $C(\bar{D})$.

\medskip

\paragraph{\textbf{Notation}}
Given $v\in \bar{\mathcal{S}}$, $f\in L_2(\Omega(v))$, and $i\in\{1,2\}$, we denote the restriction of $f$ to $\Omega_i(v)$ by $f_i$; that is, $f_i := f|_{\Omega_i(v)}$.

Throughout the paper, $c$ and $(c_k)_{k\ge 1}$ denote positive constants depending only on $L$, $H$, $d$, $V$, $\sigma_1$, and $\sigma_2$. The dependence upon additional parameters will be indicated explicitly.

%%%%%%%%%%%%%%%%
%%%%%%%%%%%%%%%%
\section{The Boundary Values}\label{sec.abv}
%%%%%%%%%%%%%%%%
%%%%%%%%%%%%%%%%

We state the precise assumptions on the function $h_v$ occurring in \eqref{a1c}. Roughly speaking, we assume that it is the trace on $\partial\Omega(v)$ of a function $h_v\in H^1(\Omega(u))$ which is such that $h|_{\Omega_i(v)}$ belongs to $H^2(\Omega_i(v))$ for $i=1,2$ and satisfies the transmission conditions \eqref{a1b}, as well as suitable boundedness and continuity properties with respect to $u$. 

Specifically, for every $v\in \bar{\mathcal{S}}$, let 
$$
h_v: D\times (-H,\infty)\rightarrow \R
$$ 
be such that
\begin{subequations}\label{bobbybrown}
\begin{equation}\label{200}
h_v\in H^1(\Omega(v))\,,\qquad h_{v,i}:=h_v\vert_{\Omega_i(v)}\in H^2\big(\Omega_i(v)\big)\,,\quad i=1,2\,,
\end{equation}
and suppose that $h_v$ satisfies the transmission conditions
\begin{equation}\label{201}
	\llbracket h_v\rrbracket=\llbracket \sigma\nabla h_v\rrbracket\cdot {\bf n}_{\Sigma(v)}=0 \ \text{ on }\ \Sigma(v)\,.
\end{equation}
For $\kappa>0$ given, there is $c(\kappa)>0$ such that, for all $v\in \bar{\mathcal{S}}$ satisfying $\|v\|_{H^2(D)}\le \kappa$,
\begin{equation}\label{202}
\| h_{v,i}\|_{H^2(\Omega_i(v))}\le c(\kappa)\,, \quad i=1,2\,.
\end{equation}
Moreover, given $v\in\bar{\mathcal{S}}$ and a sequence $(v_n)_{n\ge 1}$ in $\bar{\mathcal{S}}$ satisfying  
\begin{equation*}
\lim_{n\to \infty} \|v_n - v\|_{H^1(D)} = 0\,,
\end{equation*}
we assume that
\begin{equation}\label{204}
\lim_{n\rightarrow\infty} \| h_{v_n} -  h_{v}\|_{H^1(D\times (-H,M))}=0
\end{equation}
and 
\begin{equation}\label{205}
\lim_{n\rightarrow \infty} \|h_{v_n}(\cdot, v_n+d) -  h_{v}(\cdot, v+d)\|_{C(\bar{D})} = 0\,,
\end{equation}
where
\begin{equation*}
	M := d + \max\left\{ \|v\|_{L_\infty(D)} \,,\, \sup_{n\ge 1}\{\|v_n\|_{L_\infty(D)}\} \right\}<\infty\,.
\end{equation*}
\end{subequations}
Observe that the convergence of $(v_n)_{n\ge 1}$, the continuous embedding of $H^1(D)$ in $C(\bar{D})$, and \eqref{204} imply that
\begin{equation}\label{203}
	\lim_{n\rightarrow\infty} \int_{\Omega(v_n)}\sigma\vert\nabla h_{v_n}\vert^2\,\rd (x,z)= \int_{\Omega(v)}\sigma\vert\nabla h_{v}\vert^2\,\rd (x,z)\,.
\end{equation} 

From now on, we impose the conditions~\eqref{bobbybrown} throughout.

\medskip

We finish this short section by providing an example of $h_v$ satisfying the imposed conditions~\eqref{bobbybrown}.

%%%%%%%%%%%%%%%%
\begin{example}\label{ex1}
Let $\zeta\in C^2(\R)$ be such that $\zeta|_{(-\infty,1]}\equiv 0$ and $\zeta|_{[1+d,\infty)} \equiv V$ for some $V>0$. Given $v\in\bar{\mathcal{S}}$, put
\begin{equation}\label{exx1}
h_v(x,z) :=\zeta(z- v(x)+1)\,, \qquad -H \le z \,, \quad x\in \bar{D}\,.
\end{equation}
Then \eqref{200}-\eqref{205} are satisfied. In addition, 
$$
h_v(x,-H)=0\,, \quad h_v(x,v(x)+d)=V\,,\qquad x\in D\,.
$$
In the context of a MEMS device alluded to in the introduction, these additional properties mean that the ground plate and the top of the elastic plate are kept at constant potential. For instance, $\zeta(r):=V\min\{1,(r-1)^2/d^2\}$ for $r>1$ and $\zeta\equiv 0$ on $(-\infty,1]$ will do.
\end{example}
%%%%%%%%%%%%%%%%

%%%%%%%%%%%%%%%%
%%%%%%%%%%%%%%%%
\section{Variational Solution to \eqref{psi}}\label{sec.vs}
%%%%%%%%%%%%%%%%
%%%%%%%%%%%%%%%%

In this section we investigate the properties of the variational solution $\psi_v$ to~\eqref{psi} for~$v\in \bar{\mathcal{S}}$ and, in particular, its $H^1$-stability.

%%%%%%%%%%%%%%%%
%%%%%%%%%%%%%%%%
\subsection{A Variational Approach to \eqref{psi}}\label{sec.min}
%%%%%%%%%%%%%%%%
%%%%%%%%%%%%%%%%

Given $v\in \bar{\mathcal{S}}$ we introduce the set of admissible potentials
$$
\mathcal{A}(v):=h_{v}+H_{0}^1(\Omega(v))\,,
$$
on which we define the functional
\begin{equation}\label{sos}
	\mathcal{J}(v)[\theta]:=\frac{1}{2}\int_{ \Omega(v)} \sigma \vert\nabla \theta\vert^2\,\rd (x,z)\,,\qquad \theta\in \mathcal{A}(v)\,.
\end{equation}

The variational solution $\psi_v$  to the transmission problem \eqref{psi} is then the minimizer of the functional $\mathcal{J}(v)$ on the set $\mathcal{A}(v)$:

%%%%%%%%%%%%%%%%
\begin{lemma}\label{L1}
For each $v\in \bar{\mathcal{S}}$ there is a unique minimizer $\psi_v \in \mathcal{A}(v)$ of $\mathcal{J}(v)$ on $\mathcal{A}(v)$; that is, 
\begin{equation}
	\mathcal{J}(v)[\psi_v] = \min_{\theta\in \mathcal{A}(v)} \mathcal{J}(v)[\theta]\,. \label{z0}
\end{equation}
In addition,
\begin{equation}
\int_{\Omega(v)}\sigma\vert\nabla\psi_v\vert^2\,\rd(x,z)\le \int_{\Omega(v)}\sigma\vert\nabla h_v\vert^2\,\rd(x,z)\,. \label{z1}
\end{equation}
\end{lemma}
%%%%%%%%%%%%%%%%

\begin{proof}
Let $v\in \bar{\mathcal{S}}$ and recall that $h_v\in H^1(\Omega(v))$ according to \eqref{200}. Thus, the existence of a minimizer $\psi_v$ of $\mathcal{J}(v)$ on $\mathcal{A}(v)$ readily follows from the direct method of calculus of variations due to the lower semicontinuity and coercivity of $\mathcal{J}(v)$ on $\mathcal{A}(v)$, the latter being ensured by the assumption $\sigma\ge \min\{\sigma_1,\sigma_2\}>0$ and Poincar\'e's inequality.  The uniqueness of $\psi_v$ is guaranteed by the strict convexity of $\mathcal{J}(v)$. Next, since obviously $h_v\in \mathcal{A}(v)$, the inequality \eqref{z1} is an immediate consequence of the minimizing property \eqref{z0} of $\psi_v$.
\end{proof}

For further use, we report the following version of Poincar\'e's inequality for functions in $H^1_0(\Omega(v))$ with a constant depending mildly on $v\in\bar{\mathcal{S}}$.

%%%%%%%%%%%%%%%%
\begin{lemma}\label{LP}
Let $v\in\bar{\mathcal{S}}$ and $\theta\in H_0^1(\Omega(v))$. Then
\begin{equation*}
	\|\theta\|_{L_2(\Omega(v))} \le 2 \|H+d+v\|_{L_\infty(D)} \|\partial_z \theta\|_{L_2(\Omega(v))}\,.
\end{equation*}
\end{lemma}
%%%%%%%%%%%%%%%%

\begin{proof}
For $x\in D$ and $z\in (-H,v(x)+d)$,
\begin{equation*}
	\theta(x,z)^2 = 2 \int_{-H}^z \theta(x,y) \partial_z\theta(x,y)\ \mathrm{d}y\,.
\end{equation*}
Hence, after integration with respect to $(x,z)$ over $\Omega(v)$,
\begin{align*}
	\|\theta\|_{L_2(\Omega(v))}^2 & = \int_{\Omega(v)} \theta(x,z)^2\ \mathrm{d}(x,z) \\
	& \le 2 \|H+d+v\|_{L_\infty(D)} \int_{\Omega(v)} |\theta(x,y)| |\partial_z \theta(x,y)|\ \mathrm{d}(x,z) \\
	& \le 2 \|H+d+v\|_{L_\infty(D)} \|\theta\|_{L_2(\Omega(v))} \|\partial_z \theta\|_{L_2(\Omega(v))}\,,
\end{align*}
from which we deduce the stated inequality.
\end{proof}

%%%%%%%%%%%%%%%%
%%%%%%%%%%%%%%%%
\subsection{$H^1$-Stability of $\psi_v$}\label{sec.h1s}
%%%%%%%%%%%%%%%%
%%%%%%%%%%%%%%%%

The purpose of this section is to study the continuity properties of the solution $\psi_v$ to \eqref{z0} with respect to $v$. More precisely, we aim at establishing the following result.

%%%%%%%%%%%%%%%%
\begin{proposition}\label{C3}
Consider $v\in \bar{\mathcal{S}}$ and a sequence $(v_n)_{n\ge 1}$ in $\bar{\mathcal{S}}$ such that 
\begin{equation}\label{o1}
	v_n\rightarrow v \ \text{ in }\ H_0^1(D)\,,
\end{equation}
and set
\begin{equation}
	M := d + \max\left\{ \|v\|_{L_\infty(D)} \,,\, \sup_{n\ge 1}\{\|v_n\|_{L_\infty(D)}\}  \right\}\,, \label{z10}
\end{equation}
which is finite by \eqref{o1} and the continuous embedding of $H^1(D)$ in $C(\bar{D})$. Then
$$
\lim_{n\to\infty} \left\| (\psi_{v_n}-h_{v_n}) - (\psi_v - h_v) \right\|_{H_0^1(D\times (-H,M))} = 0
$$
and
$$
\lim_{n\to\infty} \mathcal{J}(v_n)[\psi_{v_n}] = \mathcal{J}(v)[\psi_v]\,.
$$
\end{proposition}
%%%%%%%%%%%%%%%%

To prove Proposition~\ref{C3}, we make use of a $\Gamma$-convergence approach and argue as in \cite[Section~3.2]{ARMA20} with minor changes. For the sake of completeness we  provide a complete proof in Appendix~\ref{sec.pp3.3}. % \cgr We thus omit the proof here and refer to \cite{} for details.

%%%%%%%%%%%%%%%%
%%%%%%%%%%%%%%%%
\section{$H^2$-Regularity}\label{sec.h2r}
%%%%%%%%%%%%%%%%
%%%%%%%%%%%%%%%%

In the previous section we introduced the variational solution $\psi_v\in H^1(\Omega(v)$ to \eqref{psi} for arbitrary $v\in\bar{\mathcal{S}}$ and noticed its continuous dependence in $H^1(\Omega(v)$ with respect to $v$. We now aim at improving the $H^1$-regularity of $\psi_v|_{\Omega_i(v)}$ to $H^2(\Omega_i(v))$ for $i=1,2$. To this end we first consider the case of smooth functions $v\in \mathcal{S} \cap W_\infty^2(D)$ with empty coincidence sets and provide in Section~\ref{sec.h2r1}  and Section~\ref{sec.h2e}  the corresponding $H^2$-estimates that depend only on the norm of $v$ in $H^2(D)$ (but {\it not} on its $W_\infty^2(D)$-norm). In Section~\ref{sec.h2e1} we extend these estimates to the general case $v\in\bar{\mathcal{S}}$ by means of a compactness argument.

%%%%%%%%%%%%%%%%
%%%%%%%%%%%%%%%%
\subsection{$H^2$-Regularity for $v\in \mathcal{S}\cap W_\infty^2(D)$}\label{sec.h2r1}
%%%%%%%%%%%%%%%%
%%%%%%%%%%%%%%%%

Assuming that $v$ is smoother with an empty coincidence set, see Figure~\ref{Fig1}, the existence of a strong solution $\psi_v$ to \eqref{psi} is a consequence of the analysis performed in \cite{Lem77}.

%%%%%%%%%%%%%%%%
\begin{proposition}\label{prz1}
	If $v\in \mathcal{S} \cap W_\infty^2(D)$, then the variational solution $\psi_v$ to \eqref{z0} satisfies 
	\begin{equation*}
		\psi_{v,i} := \psi_v|_{\Omega_i(v)} \in H^2(\Omega_i(v))\,, \quad i=1,2\,,
	\end{equation*} 
	and the transmission problem
	\begin{subequations}\label{psiS}
		\begin{align}
			\mathrm{div}(\sigma\nabla\psi_v)&=0 \quad\ \text{in }\ \Omega(v)\,,\label{a1aS}\\
			\llbracket \psi_v \rrbracket =\llbracket \sigma\nabla \psi_v \rrbracket \cdot \mathbf{n}_{ \Sigma(v)} &=0\quad\ \text{on }\ \Sigma(v)\,,\label{a1bS}\\
			\psi_v&=h_v\quad\text{on }\ \partial\Omega(v)\,.\label{a1cS}
		\end{align}
	\end{subequations}
	Moreover, $\partial_x \psi_v + \partial_x v \partial_z \psi_v$ and $-\sigma\partial_x v \partial_x \psi_v +\sigma \partial_z\psi_v$ both belong to $H^1(\Omega(v))$.
\end{proposition}
%%%%%%%%%%%%%%%%

Besides \cite{Lem77}, the proof of Proposition~\ref{prz1} requires the following auxiliary result.

%%%%%%%%%%%%%%%%
\begin{lemma}\label{L0}
Let $v\in\bar{\mathcal{S}}$ and consider $\phi\in L_2(\Omega(v))$ such that
\begin{equation*}
	\phi_i := \phi|_{\Omega_i(v)} \in H^1(\Omega_i(v))\,, \quad i=1,2\,,
\end{equation*}
and $\llbracket \phi\rrbracket=0$ on $\Sigma(v)$. Then $\phi\in H^1(\Omega(v))$ and
\begin{equation}
	\|\phi\|_{H^1(\Omega(v))} \le \|\phi_1\|_{H^1(\Omega_1(v))} + \|\phi_2\|_{H^1(\Omega_2(v))} \,. \label{z2}
\end{equation}
\end{lemma}
%%%%%%%%%%%%%%%%

\begin{proof}
We set $e_x=(1,0)$ and $e_z=(0,1)$. Given $\theta\in C_c^\infty\big(\Omega(v)\big)$ and $j\in \{x,z\}$ we note that
\begin{equation*}
\begin{split}
\int_{\Omega(v)}\phi \partial_j\theta \,\rd (x,z) & = \int_{\Omega(v)} \mathrm{div}(\phi\theta e_j)\, \rd (x,z) - \sum_{i=1}^2 \int_{\Omega_i(v)} \theta \partial_j \phi_i\,\rd (x,z)\\
&=\int_{\Sigma(v)} \llbracket \phi\rrbracket\,\theta e_j\cdot \mathbf{n}_{ \Sigma(v)}\,\rd \sigma_{\Sigma(v)} - \sum_{i=1}^2 \int_{\Omega_i(v)} \theta \partial_j \phi_i\,\rd (x,z)\,,
\end{split}
\end{equation*}
due to Gau\ss' theorem. Thus, since $\llbracket \phi\rrbracket=0$ on $\Sigma(v)$,
$$
\left\vert \int_{\Omega(v)}\phi \partial_j\theta\,\rd (x,z)\right\vert\le\big(\Vert\phi_1\Vert_{H^1(\Omega_1(v))}+\Vert\phi_2\Vert_{H^1(\Omega_2(v))}\big)\,\Vert\theta\Vert_{L_2(\Omega(v))}\,,
$$
for $j=x,z$ and $\theta\in C_c^\infty\big(\Omega(v)\big)$.
Consequently, $\phi\in H^1(\Omega(v))$.
\end{proof}

\begin{proof}[Proof of Proposition~\ref{prz1}]
We check that the transmission problem \eqref{psiS} fits into the framework of \cite{Lem77}. Since $v\in \mathcal{S} \cap W_\infty^2(D)$ and $v(\pm L) = 0$, the boundaries of $\Omega_1(v)$ and $\Omega_2(v)$ are $W_\infty^2$-smooth curvilinear polygons and the interface $\Sigma(v)$ meets the boundary $\partial\Omega(v)$ of $\Omega(v)$ at the vertices $A_\pm := (\pm L,0)$. Moreover, at the vertex $A_\pm$, the measures $\omega_{\pm,1}$ and $\omega_{\pm,2}$ of the angles between $-e_z$ and $(1,\mp \partial_x v(\pm L))$ and between $(1,\mp \partial_x v(\pm L))$ and $e_z$, respectively, satisfy $\omega_{\pm,1}+\omega_{\pm,2} = \pi$, as well as 
\begin{equation*}
	\begin{split}
	\omega_{\pm,2} & \ge \frac{\pi}{2} \;\text{ if }\; \llbracket \sigma \rrbracket < 0\,, \\
	\omega_{\pm,2} & \le \frac{\pi}{2} \;\text{ if }\; \llbracket \sigma \rrbracket > 0\,,
\end{split}
\end{equation*}
by definition of $\mathcal{S}$. According to the analysis performed in \cite{Lem77}, these conditions guarantee that the variational solution $\psi_v$ to \eqref{z0} provided by Lemma~\ref{L1} satisfies $\psi_{v,i} = \psi_v|_{\Omega_i(v)} \in H^2(\Omega_i(v))$ for  $i=1,2$ and solves the transmission problem \eqref{psi} in a strong sense. 

Next, owing to the just established $H^2$-regularity of $\psi_{v,1}$ and $\psi_{v,2}$, we may differentiate with respect to $x$ the transmission condition $\llbracket \psi_v \rrbracket(x,v(x))=0$, $x\in D$, and find that
\begin{equation*}
	\llbracket \partial_x \psi_v + \partial_x v \partial_z \psi_v \rrbracket = 0 \quad\ \text{ on }\ \Sigma(v)\,.
\end{equation*}
The stated $H^1$-regularity of $\partial_x \psi_v + \partial_x v \partial_z \psi_v$ then follows from Lemma~\ref{L0} and the boundedness of $\partial_x v$ and $\partial_x^2 v$. In the same vein, due to \eqref{a1b}, the regularity of $v$, and the identity
\begin{equation*}
	 \frac{\llbracket - \sigma \partial_x v \partial_x \psi_v + \sigma \partial_z \psi_v \rrbracket}{\sqrt{1+(\partial_x v)^2}} = \llbracket \sigma\nabla \psi_v \rrbracket \cdot \mathbf{n}_{ \Sigma(v)} = 0\,,
\end{equation*}
the claimed $H^1$-regularity of $- \sigma \partial_x v \partial_x \psi_v + \sigma \partial_z \psi_v$ is again a consequence of Lemma~\ref{L0} and the boundedness of $\partial_x v$ and $\partial_x^2 v$.
\end{proof}

%%%%%%%%%%%%%%%%
%%%%%%%%%%%%%%%%
\subsection{$H^2$-Estimates on $\psi_v$ for $v\in \mathcal{S} \cap W_\infty^2(D)$}\label{sec.h2e}
%%%%%%%%%%%%%%%%
%%%%%%%%%%%%%%%%

The $H^2$-regularity of $\psi_v$ being guaranteed by Proposition~\ref{prz1} for $v\in \mathcal{S}\cap W_\infty^2(D)$, the next step is to show that this property extends to any $v\in\bar{\mathcal{S}}$. To this end, we shall now derive quantitative $H^2$-estimates on $\psi_v$, paying special attention to their dependence upon the regularity of $v$. As in \cite{ARMA20}, it turns out to be more convenient to study a non-homogeneous transmission problem with homogeneous Dirichlet boundary conditions instead of \eqref{psiS}. Specifically, for $v\in \mathcal{S}\cap W_\infty^2(D)$, we define
\begin{equation}
\chi=\chi_v:=\psi_v-h_v\in H_{0}^1(\Omega(v))\,, \label{chi0}
\end{equation}
where $\psi_v\in H^1(\Omega(v))$ is the unique solution to \eqref{psiS} provided by Proposition~\ref{prz1}. Since $\psi_{v,i}=\psi_v|_{\Omega_i(v)}$ belongs to $H^2(\Omega_i(v))$ for $i=1, 2$, we readily infer from \eqref{200} and \eqref{chi0} that
\begin{equation}\label{chi}
	\chi_{i} := \chi_{v}|_{\Omega_i(v)}\in H^2(\Omega_i(v))\,, \quad i=1,2\,.
\end{equation}
We omit in the following the dependence of $\chi$ on $v$ for ease of notation. 

According to \eqref{200}, \eqref{201}, and Proposition~\ref{prz1},  $\chi$ solves the transmission problem
\begin{subequations}\label{a2}
	\begin{align}
		\mathrm{div}(\sigma\nabla\chi) & = - \mathrm{div}(\sigma\nabla h_v) \quad\text{in }\ \Omega(v)\,,\label{a2a}\\
		\llbracket \chi \rrbracket = \llbracket \sigma\nabla \chi\rrbracket\cdot {\bf n}_{\Sigma(v)} &=0\quad\text{on }\ \Sigma(v)\,,\label{a2b}\\
		\chi&=0\quad\text{on }\ \partial\Omega(v)\,, \label{a2c}
	\end{align}
\end{subequations}
and it follows from \eqref{200} that it is equivalent to derive $H^2$-estimates on $(\psi_{v,1},\psi_{v,2})$ or $(\chi_1,\chi_2)$. 

For that purpose, we transform \eqref{a2} to a transmission problem on the rectangle $\mathcal{R}:=D\times (0,1+d)$. More precisely, we  introduce the transformation
\begin{equation}\label{t1}
	T_1(x,z):=\left(x,\frac{z+H}{v(x)+H}\right)\,,\qquad (x,z)\in \Omega_1(v)\,,
\end{equation}
mapping $\Omega_1(v)$ onto the rectangle \mbox{$\mathcal{R}_1:=D\times (0,1)$}, and the transformation
\begin{equation}\label{t2}
	T_2(x,z):=\left(x,z-v(x)+1\right)\,,\qquad (x,z)\in \Omega_2(v)\,,
\end{equation}
mapping $\Omega_2(v)$ onto the  rectangle \mbox{$\mathcal{R}_2:=D\times (1,1+d)$}. The interface separating $\mathcal{R}_1$ and $\mathcal{R}_2$ is
$$
\Sigma_0:=D\times \{1\}\,,
$$ 
so that
$$
\mathcal{R}=D\times (0,1+d)=\mathcal{R}_1\cup \mathcal{R}_2\cup \Sigma_0\,.
$$
It is worth pointing out here that $T_1$ is well-defined due to $v\in \mathcal{S}$. Let $(x,\eta)$ denote the new variables in $\mathcal{R}$; that is, $(x,\eta)=T_1(x,z)$ for $(x,z)\in \mathcal{R}_1$ and $(x,\eta)=T_2(x,z)$ for $(x,z)\in \mathcal{R}_2$.
Then, \eqref{chi} implies 
\begin{equation}\label{C0}
	\Phi:=\Phi_1 \mathbf{1}_{\mathcal{R}_1} + \Phi_2 \mathbf{1}_{\mathcal{R}_2}\in H_0^1(\mathcal{R})\,,\qquad
	\Phi_i:= \chi_i\circ (T_i)^{-1}\in H^2(\mathcal{R}_i)\,,\quad i=1,2\,.
\end{equation} 
For further use, we also introduce
$$
\hat\sigma(x,\eta):=\left\{\begin{array}{cl} \dfrac{\sigma_1}{v(x)+H}\,, & (x,\eta)\in \mathcal{R}_1\,,\\
\hphantom{x}\vspace{-3.5mm}\\
\sigma_2  \,, & (x,\eta)\in \mathcal{R}_2\,,
\end{array}\right.
$$
and derive the following fundamental identity for $\Phi$, which provides a connection between some integrals involving products of second-order derivatives of $\Phi$ and is in the spirit of \cite[Lemma~4.3.1.2]{Gr85}, \cite[Lemma~3.4]{ARMA20}, and \cite[Lemme~II.2.2]{Lem77}.

%%%%%%%%%%%%%%%%
\begin{lemma}\label{L2a}
Given $v\in \mathcal{S} \cap W_\infty^2(D)$, the function $\Phi$ defined in \eqref{C0} satisfies
\begin{align*}
\sum_{i=1}^2 \int_{\mathcal{R}_i} \hat{\sigma} \partial_x^2\Phi_i\, \partial_\eta^2\Phi_i\, \rd (x,\eta) & =  \sum_{i=1}^2 \int_{\mathcal{R}_i} \hat{\sigma} |\partial_{x}\partial_{\eta}\Phi_i|^2\,\rd (x,\eta) \\
& \qquad - \sigma_1 \int_{\mathcal{R}_1} \frac{\partial_x v}{(v+H)^2} \partial_\eta\Phi_1 \partial_{x}\partial_{\eta}\Phi_1\, \rd (x,\eta)\\
& \qquad +\frac{1}{2}\int_D\frac{\partial_x^2 v \big((\partial_x v)^2-1\big)}{(1+(\partial_x v)^2)^2} \, \left\llbracket\sigma(\partial_x\Phi )^2\right\rrbracket (x,1)\,\rd x\,.
\end{align*}
\end{lemma}
%%%%%%%%%%%%%%%%

\begin{proof}
We adapt the proof of \cite[Lemma~3.4]{ARMA20} and \cite[Lemme~II.2.2]{Lem77}. Note that \eqref{a2b}, \eqref{t1}, \eqref{t2}, and \eqref{C0} imply $\llbracket \Phi\rrbracket=0$ on $\Sigma_0$, so that 
\begin{equation}\label{i}
 \llbracket \partial_x\Phi\rrbracket=0\quad\text{ on }\ \Sigma_0\,.
\end{equation}
Consequently, since $(\partial_x\Phi_1,\partial_x\Phi_2)$ lies in $H^1(\mathcal{R}_1)\times H^1(\mathcal{R}_2)$ by \eqref{C0}, we may argue as in the proof of Lemma~\ref{L0} and deduce from \eqref{i} that
$$
F:=\partial_x\Phi \in H^1(\mathcal{R})\,.
$$
Moreover, by \eqref{C0},
\begin{equation}
	F(x,0) = F(x,1+d) = 0\,, \qquad x\in D\,. \label{z4}
\end{equation}
Similarly, setting
$$
G:=-\sigma\frac{\partial_x v}{1+(\partial_x v)^2}\partial_x\Phi+\hat\sigma\partial_\eta\Phi
$$
we derive from \eqref{C0} that $G_i := G|_{\mathcal{R}_i}\in H^1(\mathcal{R}_i)$ for $i=1,2$, while \eqref{a2b}, \eqref{t1}, \eqref{t2}, and \eqref{C0} imply that, for $x\in D$,
\begin{align*}
G_1(x,1) & = \frac{\sigma_1}{\sqrt{1+(\partial_x v(x))^2}} \left[ - \partial_x v(x) \partial_x \chi_1(x,v(x)) + \partial_z \chi_1(x,v(x)) \right] \\
& = \frac{\sigma_2}{\sqrt{1+(\partial_x v(x))^2}} \left[ - \partial_x v(x) \partial_x \chi_2(x,v(x)) + \partial_z \chi_2(x,v(x)) \right] = G_2(x,1)\,;
\end{align*}	
that is, $\llbracket G\rrbracket=0$ on~$\Sigma_0$, and we argue as in the proof of Lemma~\ref{L0} to conclude that
$$
G \in H^1(\mathcal{R})\,.
$$
In addition, by \eqref{C0},
\begin{align*}
	G(\pm L, \eta) & = - \sigma (\pm L,\eta) \left( \frac{\partial_x v}{1 + (\partial_x v)^2} \right)(\pm L) \partial_x \Phi(\pm L,\eta) + \hat{\sigma}(\pm L,\eta) \partial_\eta \Phi(\pm L,\eta) \\
	& = - \sigma (\pm L,\eta) \left( \frac{\partial_x v}{1 + (\partial_x v)^2} \right)(\pm L) \partial_x \Phi(\pm L,\eta) 
\end{align*}
for $\eta\in (0,1+d)$. Hence,
\begin{equation}
G(\pm L, \eta) + \sigma (\pm L,\eta) \left( \frac{\partial_x v}{1 + (\partial_x v)^2} \right)(\pm L) F(\pm L,\eta)= 0\,, \qquad \eta\in (0,1+d)\,. \label{z5}
\end{equation}
Owing to \eqref{z4}, \eqref{z5}, and the $H^1$-regularity of $F$ and $G$, we are in a position to apply Lemma~\ref{ID} (see Appendix~\ref{sec.id}) with
\begin{equation*}
	(V,W)=(F,G) \;\;\text{ and }\;\; \tau^{\pm} = \sigma \left( \frac{\partial_x v}{1+(\partial_x v)^2} \right)(\pm L)\,, 
\end{equation*}
to obtain the identity
\begin{equation}\label{DI}
\int_\mathcal{R}\partial_x F\partial_\eta G\,\rd (x,\eta)=\int_\mathcal{R}\partial_{\eta}F\partial_x G\,\rd (x,\eta)\,.
\end{equation}
Using the definitions of $F$ and $G$, the identity \eqref{DI} reads
\begin{equation*}
\begin{split}
\sum_{i=1}^2\int_{\mathcal{R}_i} \partial_x^2\Phi_i &\left(-\sigma\frac{\partial_x v}{1+(\partial_x v)^2} \partial_{x}\partial_{\eta}\Phi_i + \hat{\sigma} \partial_\eta^2 \Phi_i \right) \,\rd (x,\eta)\\
& = \sum_{i=1}^2\int_{\mathcal{R}_i}\partial_{x}\partial_{\eta}\Phi_i \left( -\sigma\frac{\partial_x v}{1+(\partial_x v)^2} \partial_{x}^2\Phi_i - \sigma \frac{\partial_x^2 v [1 -(\partial_x v)^2]}{[1+(\partial_x v)^2]^2} \partial_x\Phi_i \right)\, \rd (x,\eta) \\
& \qquad + \sum_{i=1}^2\int_{\mathcal{R}_i}\partial_{x}\partial_{\eta}\Phi_i \Big(  \partial_x\hat\sigma \partial_\eta\Phi_i +\hat\sigma\partial_{x}\partial_{\eta}\Phi_i \Big) \,\rd (x,\eta)\,.
\end{split}
\end{equation*}
Noticing that the first terms on both sides of the above identity are the same and that
$$
\partial_x\Phi_i\partial_{x}\partial_{\eta}\Phi_i=\frac{1}{2} \partial_\eta\big((\partial_x\Phi_i)^2\big)
$$
implies that
\begin{equation*}
\begin{split}
\sum_{i=1}^2\int_{\mathcal{R}_i} \sigma & \frac{\partial_x^2 v \big[(\partial_x v)^2)-1\big]}{[1+(\partial_x v)^2]^2} \partial_x\Phi_i\partial_{x}\partial_{\eta}\Phi_i \,\rd (x,\eta)\\
& =\frac{1}{2}\int_D\frac{\partial_x^2 v \big[(\partial_x v)^2-1\big]}{[1+(\partial_x v)^2]^2} \, \left\llbracket\sigma(\partial_x\Phi )^2\right\rrbracket (x,1)\,\rd x\,,
\end{split}
\end{equation*}
the assertion follows, recalling that $\partial_x \hat{\sigma}=0$ in $\mathcal{R}_2$.
\end{proof}

%%%%%%%%%%%%%%%%
\begin{remark}\label{rem.gr}
If $\partial_x v(\pm L)=0$, then \eqref{z5} reduces to $G(\pm L,\eta) = 0$ for $\eta\in (0,1+d)$ and the crucial identity~\eqref{DI} used in the proof of Lemma~\ref{L2a} directly follows from \cite[Lemma~4.3.1.2]{Gr85}. For the general case $v\in \mathcal{S}$,  we require the extension given in  Lemma~\ref{ID}.
\end{remark}
%%%%%%%%%%%%%%%%

We now translate the outcome of Lemma~\ref{L2a} in terms of the solution $\chi$ to \eqref{a2}.

%%%%%%%%%%%%%%%%
\begin{lemma}\label{L2}
Let $v\in \mathcal{S} \cap W_\infty^2(D)$. The solution $\chi=\psi_v-h_v$ to \eqref{a2} satisfies
\begin{equation*}
\begin{split}
\sum_{i=1}^2 \int_{\Omega_i(v)}\sigma\,\partial_x^2\chi_i\,\partial_z^2\chi_i\,\rd (x,z) & = \sum_{i=1}^2 \int_{\Omega_i(v)}\sigma |\partial_{x}\partial_{z}\chi_i|^2\,\rd (x,z) \\
& \qquad - \frac{\sigma_2}{2} \int_D \partial_x^2 v(x)\, \big( \partial_z\chi_2(x,v(x)+d) \big)^2\, \rd x\\
& \qquad -\frac{1}{2}\int_D\frac{\partial_x^2 v(x)}{1+(\partial_x v(x))^2}\,\left\llbracket \sigma\vert\nabla\chi\vert^2\right\rrbracket \big(x,v(x)\big)\,\rd x\,.
\end{split}
\end{equation*}
\end{lemma}
%%%%%%%%%%%%%%%%

\begin{proof}
Let us first recall the regularity of $\Phi$ stated in \eqref{C0} which validates the subsequent computations. Using the transformations $T_1$ and $T_2$ introduced in \eqref{t1} and \eqref{t2}, respectively, we obtain
 \begin{align*}
\sum_{i=1}^2 \int_{\Omega_i(v)}\sigma\,&\partial_x^2\chi_i\,\partial_z^2\chi_i\,\rd (x,z)\\
=& \int_{\mathcal{R}_1}\frac{\sigma_1}{v+H} \bigg[\partial_x^2\Phi_1 + \eta \Big( 2\Big(\frac{\partial_x v}{v+H}\Big)^2 - \frac{\partial_x^2 v}{v+H}\Big) \, \partial_\eta\Phi_1 - 2 \eta \frac{\partial_x v}{v+H}\partial_{x}\partial_{\eta}\Phi_1\\
&\hspace{4cm} +\eta^2\Big(\frac{\partial_x v}{v+H}\Big)^2\partial_\eta^2\Phi_1\bigg] \,\partial_\eta^2\Phi_1\,\rd (x,\eta)\\
&+\int_{\mathcal{R}_2}\sigma_2\Big[\partial_x^2\Phi_2-2\partial_xv\partial_{x}\partial_{\eta}\Phi_2-\partial_x^2 v\partial_\eta\Phi_2 +(\partial_x v)^2 \partial_\eta^2\Phi_2 \Big] \, \partial_\eta^2\Phi_2\,\rd (x,\eta)  \\
=& \sum_{i=1}^2 \int_{\mathcal{R}_i} \hat{\sigma}  \partial_x^2\Phi_i \,\partial_\eta^2\Phi_i\,\rd (x,\eta)\\
& +\int_{\mathcal{R}_1}\frac{\sigma_1}{v+H} \bigg[ \eta \Big( 2 \Big(\frac{\partial_x v}{v+H}\Big)^2 - \frac{\partial_x^2 v}{v+H} \Big) \, \partial_\eta\Phi_1 - 2\eta\frac{\partial_x v}{v+H} \partial_{x}\partial_{\eta}\Phi_1\\
&\hspace{4cm} +\eta^2\Big(\frac{\partial_x v}{v+H}\Big)^2\partial_\eta^2\Phi_1\bigg] \,\partial_\eta^2\Phi_1\,\rd (x,\eta)\\
&+\int_{\mathcal{R}_2}\sigma_2\Big[-2\partial_xv\partial_{x}\partial_{\eta}\Phi_2-\partial_x^2 v\partial_\eta\Phi_2 +(\partial_x v)^2 \partial_\eta^2\Phi_2 \Big] \, \partial_\eta^2\Phi_2\,\rd (x,\eta)\,.
\end{align*}
 We use Lemma~\ref{L2a} to express the first integral on the right-hand side and get
\begin{align}
& \sum_{i=1}^2 \int_{\Omega_i(v)}\sigma\, \partial_x^2\chi_i\,\partial_z^2\chi_i\,\rd (x,z) \nonumber \\
& \qquad = \int_{\mathcal{R}_1} \hat{\sigma} |\partial_{x}\partial_{\eta}\Phi_1|^2 \,\rd (x,\eta)  + \int_{\mathcal{R}_2} \hat{\sigma} |\partial_{x}\partial_{\eta}\Phi_2|^2 \,\rd (x,\eta) \nonumber \\
& \qquad\qquad +\int_{\mathcal{R}_1} \frac{\sigma_1}{v+H} \bigg[ -\frac{\partial_x v}{v+H} \partial_{\eta}\Phi_1\partial_{x}\partial_{\eta}\Phi_1 - 2 \eta \frac{\partial_x v}{v+H} \partial_{x}\partial_{\eta}\Phi_1 \partial_\eta^2\Phi_1 \nonumber \\
&\hspace{4cm}   + \eta^2 \Big(\frac{\partial_x v}{v+H}\Big)^2 \big\vert\partial_\eta^2\Phi_1\big\vert^2 + 2\eta \Big(\frac{\partial_x v}{v+H}\Big)^2 \,\partial_\eta\Phi_1 \partial_\eta^2\Phi_1 \nonumber \\
& \hspace{8cm} - \eta \frac{\partial_x^2 v}{v+H}\,\partial_\eta\Phi_1 \partial_\eta^2\Phi_1 \bigg] \, \,\rd (x,\eta) \nonumber \\
& \qquad\qquad +\int_{\mathcal{R}_2}\sigma_2\Big[-2\partial_x v \partial_{x}\partial_{\eta}\Phi_2 \partial_\eta^2\Phi_2-\partial_x^2 v\partial_\eta\Phi_2\partial_\eta^2\Phi_2 +(\partial_x v)^2\big\vert\partial_\eta^2\Phi_2\big\vert^2\Big]\,\rd (x,\eta) \nonumber \\
& \qquad\qquad +\frac{1}{2}\int_D\frac{\partial_x^2 v \big((\partial_x v)^2-1\big)}{(1+(\partial_x v)^2)^2} \, \left\llbracket\sigma(\partial_x\Phi )^2\right\rrbracket (x,1)\,\rd x \label{E1}\,.
\end{align}
We then compute separately the integrals over $\mathcal{R}_i$, $i=1,2$, and begin with the contribution of $\mathcal{R}_1$. We complete the square to get 
\begin{align*}
I_1& := \int_{\mathcal{R}_1} \frac{\sigma_1}{v+H} \bigg[ \vert\partial_{x}\partial_{\eta}\Phi_1\vert^2 -\frac{\partial_x v}{v+H} \partial_{\eta}\Phi_1\partial_{x}\partial_{\eta}\Phi_1 - 2 \eta \frac{\partial_x v}{v+H} \partial_{x}\partial_{\eta}\Phi_1 \partial_\eta^2\Phi_1 \\
&\hspace{3cm}  +\eta^2 \Big(\frac{\partial_x v}{v+H}\Big)^2 \big\vert\partial_\eta^2\Phi_1\big\vert^2 + 2\eta \Big(\frac{\partial_x v}{v+H}\Big)^2 \,\partial_\eta\Phi_1 \partial_\eta^2\Phi_1 \\
& \hspace{8cm} - \eta \frac{\partial_x^2 v}{v+H}\,\partial_\eta\Phi_1 \partial_\eta^2\Phi_1 \bigg] \, \,\rd (x,\eta)  \\
& = \int_{\mathcal{R}_1} \frac{\sigma_1}{v+H}\bigg[ \big| \partial_{x}\partial_{\eta}\Phi_1 \big|^2 + \Big(\frac{\partial_x v}{v+H}\Big)^2 \big| \partial_\eta\Phi_1 \big|^2 + \eta^2\Big(\frac{\partial_x v}{v+H}\Big)^2 \big| \partial_\eta^2\Phi_1\big|^2 \\
& \hspace{4cm} -2\eta\frac{\partial_x v}{v+H} \partial_{x}\partial_{\eta}\Phi_1 \partial_\eta^2\Phi_1 + 2\eta\Big(\frac{\partial_x v}{v+H}\Big)^2\,\partial_\eta\Phi_1 \partial_\eta^2\Phi_1 \\
& \hspace{8cm} - 2 \frac{\partial_x v}{v+H} \partial_{\eta}\Phi_1 \partial_{x}\partial_{\eta}\Phi_1 \bigg]\ \rd (x,\eta)\\
& \qquad + \int_{\mathcal{R}_1} \frac{\sigma_1}{v+H} \bigg[ - \Big(\frac{\partial_x v}{v+H}\Big)^2 \big| \partial_\eta\Phi_1 \big|^2 + \frac{\partial_x v}{v+H} \partial_{\eta}\Phi_1 \partial_{x}\partial_{\eta}\Phi_1 \\
& \qquad \hspace{6cm} - \eta \frac{\partial_x^2 v}{v+H} \partial_\eta\Phi_1 \partial_\eta^2 \Phi_1\bigg]\ \rd (x,\eta) \\
& = \int_{\mathcal{R}_1} \sigma_1 (v+H) \bigg[ \frac{\partial_{x}\partial_{\eta}\Phi_1}{v+H} -\frac{\partial_x v}{(v+H)^2} \partial_{\eta}\Phi_1-\eta \frac{\partial_x v}{(v+H)^2}\, \partial_\eta^2\Phi_1\bigg]^2\,\rd (x,\eta)\\
&\qquad +\int_{\mathcal{R}_1} \sigma_1 \partial_x v\,\bigg[\frac{1}{(v+H)^2}\partial_{\eta}\Phi_1\partial_{x}\partial_{\eta}\Phi_1 -\frac{\partial_x v}{(v+H)^3}\big(\partial_\eta\Phi_1\big)^2\bigg] \,\rd (x,\eta)\\
&\qquad - \int_{\mathcal{R}_1} \sigma_1 \frac{\partial_x^2 v}{(v+H)^2}\,\eta\,\partial_\eta\Phi_1 \partial_\eta^2\Phi_1 \,\rd (x,\eta)\,.
\end{align*}
Thanks to the identities
\begin{equation*}
\frac{1}{(v+H)^2}\partial_\eta\Phi_1\partial_{x}\partial_{\eta}\Phi_1 - \frac{\partial_x v}{(v+H)^3} \big(\partial_\eta\Phi_1\big)^2  = \frac{1}{2} \partial_x \left( \left( \frac{\partial_\eta\Phi_1}{v+H} \right)^2\right)\,,
\end{equation*}
\begin{equation*}
 \partial_\eta\Phi_1\partial_\eta^2\Phi_1=\frac{1}{2} \partial_\eta \big(\partial_\eta\Phi_1\big)^2\,,%\label{e2}
\end{equation*}
and the property $\partial_\eta\Phi_1(\pm L,\eta)= 0$ for $\eta\in (0,1)$ stemming from \eqref{C0}, we may perform integration by parts in the last two integrals on the right-hand side of the previous identity and obtain
\begin{align*}
I_1 & =\int_{\mathcal{R}_1} \sigma_1 (v+H)\bigg[\frac{\partial_{x}\partial_{\eta}\Phi_1 }{v+H} -\frac{\partial_x v}{(v+H)^2} \partial_{\eta}\Phi_1-\eta \frac{\partial_x v}{(v+H)^2}\, \partial_\eta^2\Phi_1\bigg]^2\,\rd (x,\eta)\\
&\qquad -\frac{\sigma_1}{2}\int_{D}\frac{\partial_x^2 v}{(v+H)^2}\,\big(\partial_\eta\Phi_1(x,1)\big)^2\,\rd x\,.
\end{align*}
Transforming the above identity back to $\Omega_1(v)$ yields
\begin{equation}\label{R1}
I_1=\int_{\Omega_1(v)}\sigma_1\big\vert\partial_{x}\partial_{z}\chi_1\big\vert^2\,\rd (x,z) - \frac{\sigma_1}{2} \int_{D} \partial_x^2 v(x)\,\big( \partial_z\chi_1 (x,v(x))\big)^2\,\rd x\,.
\end{equation}
Next, arguing in a similar way,
\begin{align*}
I_2  := &\, \sigma_2 \int_{\mathcal{R}_2}\Big[ \big\vert \partial_x\partial_\eta \Phi_2\big|^2 -2\partial_x v \partial_{x}\partial_{\eta}\Phi_2 \partial_\eta^2\Phi_2-\partial_x^2 v\partial_\eta\Phi_2\partial_\eta^2\Phi_2 +(\partial_x v)^2\big\vert\partial_\eta^2\Phi_2\big\vert^2\Big]\,\rd (x,\eta) \\
= &\, \sigma_2 \int_{\mathcal{R}_2} \Big[\big| \partial_{x}\partial_{\eta}\Phi_2 \big|^2 - 2 \partial_x v\partial_{x}\partial_{\eta}\Phi_2 \partial_\eta^2\Phi_2+(\partial_x v)^2\big\vert\partial_\eta^2\Phi_2\big\vert^2\Big]\,\rd (x,\eta) \\
& \qquad - \frac{\sigma_2}{2}\int_{\mathcal{R}_2}\partial_x^2 v \partial_\eta \left( \partial_\eta \Phi_2 \right)^2 \,\rd (x,\eta) \\
= &\, \sigma_2 \int_{\mathcal{R}_2} \Big[\partial_{x}\partial_{\eta}\Phi_2 - \partial_x v \partial_\eta^2\Phi_2\Big]^2\,\rd (x,\eta) - \frac{\sigma_2}{2}\int_{D} \partial_x^2 v(x)\,\big( \partial_\eta\Phi_2 (x,1+d)\big)^2\,\rd x\\
& \qquad +\frac{\sigma_2 }{2}\int_{D} \partial_x^2 v(x)\,\big( \partial_\eta\Phi_2 (x,1)\big)^2\,\rd x\,.
\end{align*}
Transforming this formula back to $\Omega_2(v)$ yields
\begin{equation}\label{R2}
\begin{split}
I_2 &= \sigma_2 \int_{\Omega_2(v)} \big|\partial_{x}\partial_{z}\chi_2\big|^2\,\rd (x,z)
- \frac{\sigma_2}{2} \int_{D} \partial_x^2 v(x)\,\big( \partial_z\chi_2 (x,v(x)+d)\big)^2\,\rd x\\
&\qquad + \frac{\sigma_2}{2}\int_{D} \partial_x^2 v(x)\,\big( \partial_z\chi_2 (x,v(x))\big)^2\,\rd x \,.
\end{split}
\end{equation}
Finally,
\begin{align*}
	& \int_D\frac{\partial_x^2 v \big[(\partial_x v)^2-1\big]}{[1+(\partial_x v)^2]^2} \, \left\llbracket\sigma(\partial_x\Phi )^2\right\rrbracket (x,1)\,\rd x \\
	& \qquad = \int_D\frac{\partial_x^2 v \big[(\partial_x v)^2-1\big]}{[1+(\partial_x v)^2]^2} \, \left\llbracket \sigma(\partial_x\chi + \partial_x v \partial_z \chi )^2\right\rrbracket (x,1)\,\rd x\,,
\end{align*}
and we deduce from \eqref{E1}, \eqref{R1}, \eqref{R2}, and the above identity that
\begin{equation}\label{E2}
\begin{split}
\sum_{i=1}^2 \int_{\Omega_i(v)}\sigma\,&\partial_x^2\chi_i\,\partial_z^2\chi_i\,\rd (x,z)\\
=& \sum_{i=1}^2 \int_{\Omega_i(v)}\sigma\,\big|\partial_{x}\partial_{z}\chi_i\big|^2\,\rd (x,z)
-\frac{\sigma_2}{2}\int_{D} \partial_x^2 v(x)\,\big( \partial_z\chi_2 (x,v(x)+d)\big)^2\,\rd x\\
&-\frac{1}{2}\int_{D} \partial_x^2 v(x)\,\left\llbracket\sigma \big(\partial_z\chi_2\big)^2\right\rrbracket (x,v(x))\,\rd x\\
&+\frac{1}{2}\int_D\frac{\partial_x^2 v \big[(\partial_x v)^2-1\big]}{[1+(\partial_x v)^2]^2} \, \left\llbracket\sigma\big(\partial_x\chi+\partial_x v\partial_z\chi\big)^2\right\rrbracket (x,v(x))\,\rd x\,.
\end{split}
\end{equation}
It remains to simplify the last two integrals on the right-hand side of \eqref{E2}. To this end, we first recall that the regularity of $\chi$ allows us to differentiate with respect to $x$ the transmission condition $\llbracket \chi\rrbracket =0$ on $\Sigma(v)$ to deduce that
\begin{equation}
	\llbracket \partial_x \chi + \partial_x v \partial_z \chi \rrbracket = 0 \;\text{ on }\; \Sigma(v)\,, \label{z6}
\end{equation}
while the second transmission condition in \eqref{a2b} reads
\begin{equation}
	\llbracket \sigma \big( \partial_x v \partial_x \chi - \partial_z \chi \big) \rrbracket = 0 \;\text{ on }\; \Sigma(v)\,. \label{z7}
\end{equation}
In particular, \eqref{z6} and \eqref{z7} imply that, on $\Sigma(v)$,
\begin{align*}
	\llbracket \sigma \big( \partial_x v \partial_x \chi - \partial_z \chi \big) \big( \partial_x \chi + \partial_x v \partial_z \chi \big)  \rrbracket & = \big( \partial_x \chi_1 + \partial_x v \partial_z \chi_1 \big) \llbracket \sigma \big( \partial_x v \partial_x \chi - \partial_z \chi \big) \rrbracket \\
	& \qquad + \sigma_2 \big( \partial_x v \partial_x \chi_2 - \partial_z \chi_2 \big) \llbracket \big( \partial_x \chi + \partial_x v \partial_z \chi \big) \rrbracket \\
	& = 0\,.
\end{align*}
Therefore,
\begin{align*}
	J & := \big[ (\partial_x v)^2-1 \big] \Big\llbracket \sigma \big(\partial_x\chi+\partial_x v\partial_z\chi\big)^2 \Big\rrbracket - [1+(\partial_x v)^2]^2 \Big\llbracket \sigma \big( \partial_z \chi \big)^2 \Big\rrbracket  \\
	& = \big[ (\partial_x v)^2-1 \big] \Big\llbracket \sigma \big(\partial_x\chi+\partial_x v\partial_z\chi\big)^2 \Big\rrbracket - [1+(\partial_x v)^2]^2 \Big\llbracket \sigma \big( \partial_z \chi \big)^2 \Big\rrbracket  \\
	& \qquad - 2\partial_x v \Big\llbracket \sigma \big( \partial_x v \partial_x \chi - \partial_z \chi \big) \big( \partial_x \chi + \partial_x v \partial_z \chi \big)  \Big\rrbracket \\
	& = \Big\llbracket \sigma \big[ (\partial_x v)^2 - 1 - 2(\partial_x v)^2 \big] \big( \partial_x\chi \big)^2 \Big\rrbracket \\
	& \qquad + \Big\llbracket \sigma \big[ 2 \partial_x v \big( (\partial_x v)^2-1 \big) - 2 (\partial_x v)^3 + 2 \partial_x v \big] \partial_x \chi \partial_z \chi \Big\rrbracket \\
	& \qquad + \Big\llbracket \sigma \big[ (\partial_x v)^2 \big( (\partial_x v)^2 - 1\big) + 2 (\partial_x v)^2 - [1+(\partial_x v)^2]^2 \big] \big( \partial_z \chi \big)^2 \Big\rrbracket \\
	& = - \big[ 1 + (\partial_x v)^2 \big] \Big\llbracket \sigma \big( \partial_x \chi \big)^2 + \sigma \big( \partial_z \chi \big)^2 \Big\rrbracket \\
	& = - \big[ 1 + (\partial_x v)^2 \big] \Big\llbracket \sigma |\nabla\chi|^2 \Big\rrbracket\,.
\end{align*}
Hence,
\begin{equation}
\frac{(\partial_x v)^2-1}{[1+(\partial_x v)^2]^2} \Big\llbracket \sigma \big(\partial_x\chi+\partial_x v\partial_z\chi\big)^2 \Big\rrbracket -  \Big\llbracket \sigma \big( \partial_z \chi \big)^2 \Big\rrbracket = - \frac{1}{1+(\partial_x v)^2} \left\llbracket \sigma |\nabla\chi|^2 \right\rrbracket\,. \label{17}
\end{equation}
Consequently, \eqref{E2} and \eqref{17} entail 
\begin{equation*}
\begin{split}
\sum_{i=1}^2 \int_{\Omega_i(v)}\sigma\,\partial_x^2\chi_i\,\partial_z^2\chi_i\,\rd (x,z) & = \sum_{i=1}^2 \int_{\Omega_i(v)}\sigma |\partial_{x}\partial_{z}\chi_i|^2\,\rd (x,z) \\
& \quad - \frac{1}{2}\int_D \sigma_2\partial_x^2 v(x)\,\big(\partial_z\chi_2(x,v(x)+d)\big)^2\,\rd x\\
& \quad -\frac{1}{2}\int_D\frac{\partial_x^2 v(x)}{1+(\partial_x v(x))^2}\,\left\llbracket \sigma |\nabla\chi|^2\right\rrbracket \big(x,v(x)\big)\,\rd x\,,
\end{split}
\end{equation*}
as claimed.
\end{proof}

In order to estimate the boundary and the transmission terms in Lemma~\ref{L2}, we first report the following trace estimates.

%%%%%%%%%%%%%%%%%
\begin{lemma}\label{L3}
Given $\kappa>0$ and  $\alpha\in (0,1/2]$, there  is $c(\alpha,\kappa)>0$ such that, for any $v\in\bar{\mathcal{S}}$ satisfying $\|v\|_{H^2(D)}\le \kappa$ and $\theta\in H^1(\Omega_2(v))$,
\begin{equation*}
\|\theta_2(\cdot,v)\|_{H^\alpha(D)}+ \|\theta_2(\cdot,v+d)\|_{H^\alpha(D)}\le c(\alpha,\kappa)\, \|\theta_2\|_{L_2(\Omega_2(v))}^{(1-2\alpha)/2}\, \|\theta_2\|_{H^1(\Omega_2(v))}^{(2\alpha+1)/2}\,.
\end{equation*}
\end{lemma}
%%%%%%%%%%%%%%%%

\begin{proof}
Let $\theta\in H^1(\Omega_2(v))$. Using the transformation $T_2$ defined in \eqref{t2} which maps $\Omega_2(v)$ onto the rectangle $\mathcal{R}_2 =D\times (1,1+d)$, we note that $\phi:=\theta\circ T_2^{-1}$ belongs to $H^1(\mathcal{R}_2)$ with
\begin{equation}
\|\phi\|_{L_2(\mathcal{R}_2)} = \|\theta\|_{L_2(\Omega_2(v))} \label{n2}
\end{equation}
and
\begin{equation*}
	\|\nabla\phi \|_{L_2(\mathcal{R}_2)}^2 = \|\partial_x \theta + \partial_x v \partial_z  \theta \|_{L_2(\Omega_2(v))}^2 + \|\partial_z \theta \|_{L_2(\Omega_2(v))}^2\,,
\end{equation*}
so that the continuous embedding of $H^2(D)$ in $W_\infty^1(D)$ and the assumed bound on $v$ readily imply that 
\begin{equation}
\|\phi\|_{H^1(\mathcal{R}_2)} \le c(\kappa) \|\theta\|_{H^1(\Omega_2(v))} \,. \label{n1}
\end{equation}
By complex interpolation, 
$$
[L_2(\mathcal{R}_2 ),H^1(\mathcal{R}_2 )]_{\alpha+1/2}\doteq H^{\alpha+1/2} (\mathcal{R}_2 )\,,
$$
from which we deduce that
$$
\|\phi\|_{H^{\alpha+1/2}(\mathcal{R}_2 )}\le c(\alpha) \|\phi\|_{L_2(\mathcal{R}_2 )}^{(1-2\alpha)/2} \|\phi\|_{H^1(\mathcal{R}_2 )}^{(2\alpha+1)/2}\,.
$$
Since $\alpha>0$, the  trace maps $H^{\alpha+1/2}(\mathcal{R}_2 )$ continuously on $H^{\alpha}(D\times \{1\})$, and we thus infer from \eqref{n2} and \eqref{n1} that
\begin{equation*}
\begin{split} 
\|\theta(\cdot,v)\|_{H^{\alpha}(D)} &=\|\phi(\cdot, 1) \|_{H^{\alpha}(D)}\le c(\alpha) \|\phi\|_{H^{\alpha+1/2}(\mathcal{R}_2)} \\ 
& \le c(\alpha)\|\phi\|_{L_2(\mathcal{R}_2 )}^{(1-2\alpha)/2} \|\phi\|_{H^1(\mathcal{R}_2 )}^{(2\alpha+1)/2}\\
&\le  c(\alpha,\kappa)\|\theta\|_{L_2(\Omega_2(v))}^{(1-2\alpha)/2}\|\theta\|_{H^1(\Omega_2(v))}^{(2\alpha+1)/2}\,.
\end{split}
\end{equation*}
The estimate for $\|\theta(\cdot,v+d)\|_{H^{\alpha}(D)} $ is proved in a similar way.
\end{proof}

Based on Lemma~\ref{L3} we are in a position to estimate the boundary and transmission terms in the identity provided by Lemma~\ref{L2}.

%%%%%%%%%%%%%%%%
\begin{lemma}\label{L4} 
Let $\zeta\in (3/4,1)$ and $\kappa>0$. There is $c(\zeta,\kappa)>0$ such that, if $v\in \mathcal{S} \cap W_\infty^2(D)$ satisfies $\|v\|_{H^2(D)}\le \kappa$, then the solution $\chi=\chi_v$ to \eqref{a2} satisfies
\begin{equation}\label{a}
	\begin{split}
	& \left| \frac{\sigma_2}{2} \int_D\partial_x^2 v(x) \big(\partial_z\chi_2(x,v(x)+d) \big)^2 \,\rd x \right| \\
	& \hspace{3cm} \le c(\zeta,\kappa) \, \|\partial_z\chi_2\|_{L_2(\Omega_2(v))}^{2(1-\zeta)} \, \|\partial_z\chi_2\|_{H^1(\Omega_2(v))}^{2\zeta}
	\end{split}
\end{equation}
and
\begin{equation}\label{b}
	\begin{split}
	& \left| \frac{1}{2} \int_D \frac{\partial_x^2 v(x)}{1+(\partial_x v(x))^2}\, \left\llbracket \sigma\vert\nabla\chi\vert^2 \right\rrbracket \big(x,v(x)\big)\,\rd x\right| \\
	& \hspace{3cm} \le c(\zeta,\kappa)\, \|\nabla\chi_2\|_{L_2(\Omega_2(v))}^{2(1-\zeta)}\, \|\nabla\chi_2\|_{H^1(\Omega_2(v))}^{2\zeta}\,.
	\end{split}
\end{equation}
\end{lemma}
%%%%%%%%%%%%%%%%

\begin{proof} 
 To prove \eqref{a}, let us first note that $H^{\zeta-1/2}(D)$ embeds continuously into $L_4(D)$. We use the Cauchy-Schwarz inequality and Lemma~\ref{L3} with $\alpha=\zeta-1/2$ and deduce
\begin{equation*}
\begin{split}
\left| \frac{\sigma_2}{2}\int_D \partial_x^2 v(x) \big(\partial_z\chi_2(x,v(x)+d)\big)^2\, \rd x\right|& \le \frac{\sigma_2}{2} \|\partial_x^2 v\|_{L_2(D)}\, \|\partial_z\chi_2(\cdot,v+d)\|_{L_4(D)}^2\\
& \le c(\kappa)\, \|\partial_z\chi_2(\cdot,v+d)\|_{H^{\zeta-1/2}(D)}^2\\
& \le c(\zeta,\kappa)\, \|\partial_z\chi_2\|_{L_2(\Omega_2(v))}^{2(1-\zeta)}\, \|\partial_z\chi_2\|_{H^1(\Omega_2(v))}^{2\zeta}\,.
\end{split}
\end{equation*}
As for \eqref{b} we obtain analogously 
\begin{align}
&\left| \frac{\sigma_2}{2}\int_D \frac{\partial_x^2 v(x)}{1+(\partial_x v(x))^2}\, \left[ \big(\partial_x\chi_2(x,v(x))\big)^2 + \big(\partial_z\chi_2(x,v(x))\big)^2 \right] \,\rd x\right|\nonumber \\
& \qquad\qquad \le \frac{\sigma_2}{2} \|\partial_x^2 v\|_{L_2(D)} \|\nabla \chi_2(\cdot,v)\|_{L_4(D)}^2 \nonumber \\
 & \qquad\qquad \le c(\zeta,\kappa)\, \|\nabla\chi_2\|_{L_2(\Omega_2(v))}^{2(1-\zeta)}\, \|\nabla\chi_2\|_{H^1(\Omega_2(v))}^{2\zeta} \label{pp}
\end{align}
and
\begin{equation}
\begin{split}
	&\left| \frac{\sigma_1}{2}\int_D \frac{\partial_x^2 v(x)}{1+(\partial_x v(x))^2}\, \left[ \big(\partial_x\chi_1(x,v(x))\big)^2 + \big(\partial_z\chi_1(x,v(x))\big)^2 \right] \,\rd x\right| \\
	& \qquad\qquad \le \frac{\sigma_1}{2} \|\partial_x^2 v\|_{L_2(D)} \|\nabla \chi_1(\cdot,v)\|_{L_4(D)}^2 \,. \label{z8}
\end{split}
\end{equation}
At this point, we use \eqref{z6} and \eqref{z7} to show that 
\begin{align*}
	\partial_x \chi_1 & = \frac{\sigma_1+\sigma_2(\partial_x v)^2}{\sigma_1\big(1+(\partial_x v)^2\big)} \partial_x \chi_2 +\frac{\llbracket\sigma\rrbracket\partial_x v}{\sigma_1\big(1+(\partial_x v)^2\big)}\partial_z\chi_2	\quad\text{on }\ \Sigma(v)\,, \\ 
	\partial_z \chi_1 & = \frac{\llbracket\sigma\rrbracket\partial_x v}{\sigma_1\big(1+(\partial_x v)^2\big)}\partial_x\chi_2 + \frac{\sigma_1+\sigma_2(\partial_x v)^2}{\sigma_1\big(1+(\partial_x v)^2\big)} \partial_z \chi_2	\quad\text{on }\ \Sigma(v)\,. 
\end{align*}
Consequently, 
\begin{align*}
	|\partial_x \chi_1| & \le \frac{\max\{\sigma_1,\sigma_2\}}{\sigma_1} \left( |\partial_x \chi_2| + |\partial_z \chi_2| \right) \quad\text{on }\ \Sigma(v)\,, \\
	|\partial_z \chi_1| & \le \frac{\max\{\sigma_1,\sigma_2\}}{\sigma_1} \left( |\partial_x \chi_2| + |\partial_z \chi_2| \right) \quad\text{on }\ \Sigma(v)\,, 
\end{align*}
so that 
\begin{equation*}
	\|\nabla \chi_1(\cdot,v)\|_{L_4(D)} \le c \|\nabla \chi_2(\cdot,v)\|_{L_4(D)}\,.
\end{equation*}
Owing to \eqref{z8} and the above inequality, we may then argue as in the proof of \eqref{pp} to conclude that 
\begin{align*}
	&\left| \frac{\sigma_1}{2}\int_D \frac{\partial_x^2 v(x)}{1+(\partial_x v(x))^2}\, \left[ \big(\partial_x\chi_1(x,v(x))\big)^2 + \big(\partial_z\chi_1(x,v(x))\big)^2 \right] \,\rd x\right| \\
	& \qquad\qquad \le c(\zeta,\kappa)\, \|\nabla\chi_2\|_{L_2(\Omega_2(v))}^{2(1-\zeta)}\, \|\nabla\chi_2\|_{H^1(\Omega_2(v))}^{2\zeta}\,,
\end{align*}
as claimed in \eqref{b}.
\end{proof}

We now gather the previous findings to deduce the following crucial $H^2$-estimate on the solution $\psi_v$ to \eqref{psiS} for $v\in \mathcal{S}\cap W_\infty^2(D)$, which only depends on the $H^2(D)$-norm of $v$ (but not on its $W_\infty^2(D)$-norm).

%%%%%%%%%%%%%%%%
\begin{proposition}\label{P2} 
Let $\kappa>0$ and $v\in\mathcal{S}\cap W_\infty^2(D)$ be such that $\|v\|_{H^2(D)}\le \kappa$. There is a constant $c_0(\kappa)>0$ such that the solution $\psi_v$ to \eqref{psiS} satisfies
\begin{subequations}\label{est}
\begin{equation}\label{est1}
\|\chi\|_{H^1(\Omega(v))} + \|\chi_1\|_{H^2(\Omega_1(v))} + \|\chi_2\|_{H^2(\Omega_2(v))}\le c_0(\kappa)\,,
\end{equation}
and
\begin{equation}\label{est2}
\|\psi_v\|_{H^1(\Omega(v))}+\|\psi_{v,1}\|_{H^2(\Omega_1(v))}+\|\psi_{v,2}\|_{H^2(\Omega_2(v))}  \le c_0(\kappa)\,,
\end{equation}
\end{subequations}
recalling that $\chi=\psi_v-h_v$ and $\chi_i = \chi|_{\Omega_i(v)}$, $i=1,2$.
\end{proposition}
%%%%%%%%%%%%%%%%

\begin{proof}
Let $v\in \mathcal{S} \cap W_\infty^2(D)$ with $\|v\|_{H^2(D)}\le \kappa$. Since $\sigma$ is constant on $\Omega_1(v)$ and on $\Omega_2(v)$, it readily follows  from \eqref{a2a} that
$$
\sum_{i=1}^2 \int_{\Omega_i(v)}\sigma |\Delta\chi_i|^2\, \rd (x,z)= \sum_{i=1}^2 \int_{\Omega_i(v)}\sigma | \Delta h_{v,i}|^2\, \rd (x,z)\,.
$$
Since
\begin{equation*}
	|\Delta\chi_i|^2 = |\partial_x^2\chi_i|^2 + |\partial_z^2\chi_i|^2 + 2 \partial_x^2\chi_i \partial_z^2 \chi_i\,, \qquad i=1,2\,,
\end{equation*}
we infer from Lemma~\ref{L2} and the above two formulas that
\begin{align*}
& \sum_{i=1}^2 \int_{\Omega_i(v)} \sigma \big\{ |\partial_x^2\chi_i|^2+ 2|\partial_{x}\partial_{z}\chi_i|^2+|\partial_z^2\chi_i|^2\big\}\,\rd (x,z)\\
& \qquad = \sum_{i=1}^2 \int_{\Omega_i(v)}\sigma |\Delta h_{v,i}|^2\,\rd (x,z) + 2 \sum_{i=1}^2  \int_{\Omega_i(v)} \sigma \left( |\partial_x\partial_z \chi_i|^2 - \partial_x^2 \chi_i \partial_z^2\chi_i \right)\, \rd (x,z) \\
& \qquad \le \sum_{i=1}^2 \int_{\Omega_i(v)}\sigma |\Delta h_{v,i}|^2\,\rd (x,z) + \sigma_2 \int_D  \partial_x^2 v(x) \,\big(\partial_z\chi_2(x,v(x)+d)\big)^2\,\rd x\\
&\qquad\qquad + \int_D \frac{\partial_x^2 v(x)}{1+(\partial_x v(x))^2}\, \left\llbracket \sigma |\nabla\chi|^2 \right\rrbracket \big(x,v(x)\big)\, \rd x\,.
\end{align*}
Using Lemma~\ref{L4} with $\zeta=7/8$, along with the identity
\begin{equation*}
	\begin{split}
	& \sum_{i=1}^2 \int_{\Omega_i(v)}\sigma\big\{|\partial_x^2\chi_i|^2+ 2|\partial_{x}\partial_{z}\chi_i|^2+|\partial_z^2\chi_i|^2\big\}\,\rd (x,z) \\
	& \hspace{2cm} = \sigma_1 \|\nabla\chi_1\|_{H^1(\Omega_1(v))}^2 + \sigma_2 \|\nabla\chi_2\|_{H^1(\Omega_2(v))}^2\,,
	\end{split}
\end{equation*} 
we further obtain
\begin{align*}
& \sigma_1 \|\nabla\chi_1\|_{H^1(\Omega_1(v))}^2  + \sigma_2 \|\nabla\chi_2\|_{H^1(\Omega_2(v))}^2\\
& \qquad\qquad \le \sum_{i=1}^2 \int_{\Omega_i(v)}\sigma |\Delta h_{v,i}|^2\,\rd (x,z) + c(\kappa)\, \|\nabla\chi_2\|_{L_2(\Omega_2(v))}^{1/4}\, \|\nabla\chi_2\|_{H^1(\Omega_2(v))}^{7/4}\,.
\end{align*}
Hence, thanks to Young's inequality,
\begin{align*}
	& \sigma_1 \|\nabla\chi_1\|_{H^1(\Omega_1(v))}^2  + \sigma_2 \|\nabla\chi_2\|_{H^1(\Omega_2(v))}^2\\
	& \qquad\qquad\le \sum_{i=1}^2 \int_{\Omega_i(v)}\sigma |\Delta h_{v,i}|^2\,\rd (x,z) + \frac{\sigma_2}{2}\, \|\nabla\chi_2\|_{H^1(\Omega_2(v))}^2 + c(\kappa) \|\nabla\chi_2\|_{L_2(\Omega_2(v))}^2\,.
\end{align*}
Recalling that 
\begin{align*}
	\|\nabla \chi_2\|_{L_2(\Omega_2(v))}^2 & \le \frac{1}{\sigma_2} \int_{\Omega(v)} \sigma |\nabla\chi|^2\, \rd(x,z) \le \frac{1}{\sigma_2} \int_{\Omega(v)} \sigma |\nabla h_v|^2\, \rd (x,z) \\
	& \le \frac{\max\{\sigma_1,\sigma_2\}}{\sigma_2} \|\nabla h_v\|_{L_2(\Omega(v))}^2
\end{align*}
by \eqref{a2} and that $\min\{\sigma_1,\sigma_2\}>0$, we conclude that 
\begin{equation}
	\begin{split}
	& \|\nabla\chi_1\|_{H^1(\Omega_1(v))}^2  + \|\nabla\chi_2\|_{H^1(\Omega_2(v))}^2 \\
	& \hspace{2cm} \le c(\kappa) \left( \|\Delta h_{v,1}\|_{L_2(\Omega_1(v))}^2 + \|\Delta h_{v,2}\|_{L_2(\Omega_2(v))}^2 + \|\nabla h_v\|_{L_2(\Omega(v))}^2 \right)\,.
	\end{split} \label{25}
\end{equation}
Owing to the continuous embedding of $H^2(D)$ in $C(\bar{D})$, combining \eqref{25} and Lemma~\ref{LP} leads us to the estimate
\begin{equation*}
	\begin{split}
		& \|\chi\|_{H^1(\Omega(v))} + \|\chi_{1}\|_{H^2(\Omega_1(v))} + \|\chi_{2}\|_{H^2(\Omega_2(v))} \\
		& \hspace{2cm} \le c(\kappa) \big(\|\nabla h_v \|_{L_2(\Omega(v))} + \|\Delta h_{v,1}\|_{L_2(\Omega_1(v))}^2 + \|\Delta h_{v,2}\|_{L_2(\Omega_2(v))}^2 \big)\,.
	\end{split}
\end{equation*}
The bound~\eqref{est1} then readily follows from the assumptions~\eqref{200} and~\eqref{202}. Finally, \eqref{est1}, together with \eqref{200} and \eqref{202}, yields~\eqref{est2}.
\end{proof}

%%%%%%%%%%%%%%%%
%%%%%%%%%%%%%%%%
\subsection{$H^2$-Regularity and $H^2$-Estimates on $\psi_v$ for $v\in \bar{\mathcal{S}}$}\label{sec.h2e1}
%%%%%%%%%%%%%%%%
%%%%%%%%%%%%%%%%

Finally, we extend Proposition~\ref{prz1} and Proposition~\ref{P2} by showing the $H^2$-regularity of $\psi_v$ and the corresponding $H^2$-estimates for an arbitrary $v\in\bar{\mathcal{S}}$; that is, we drop the additional $W_\infty^2$-regularity of $v$ assumed in the previous sections and also allow for a non-empty coincidence set.
%%%%%%%%%%%%%%%%
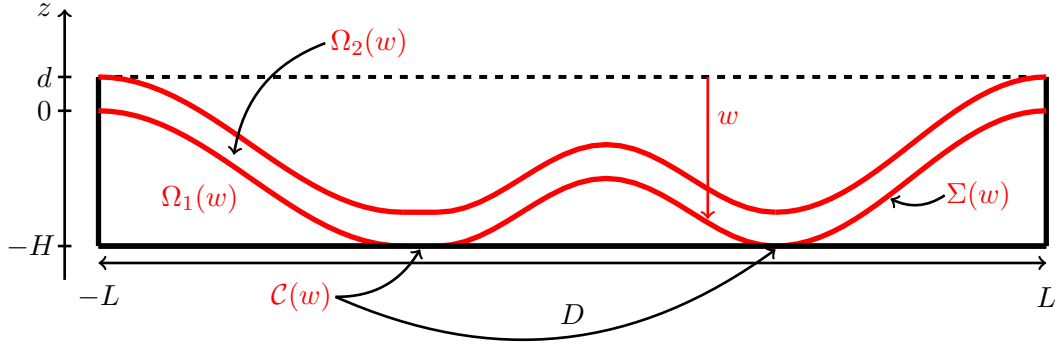
\begin{figure}
	\begin{tikzpicture}[scale=0.9]
		\draw[black, line width = 1.5pt, dashed] (-7,0)--(7,0);
		\draw[black, line width = 2pt] (-7,0)--(-7,-2.5);
		\draw[black, line width = 2pt] (7,-2.5)--(7,0);
		\draw[black, line width = 2pt] (-7,-2.5)--(7,-2.5);
		\draw[red, line width = 2pt] plot[domain=-7:-2.5] (\x,{-1-cos((pi*(11.5+\x)/4.5) r)});
		\draw[red, line width = 2pt] plot[domain=-7:-2.5] (\x,{-1.5-cos((pi*(11.5+\x)/4.5) r)});
		\draw[red, line width = 2pt] (-2.5,-2)--(-2,-2);
		\draw[red, line width = 2pt] plot[domain=-2:3] (\x,{-2.0-0.5*cos((pi*(4+2*\x)/5) r)});
		\draw[red, line width = 2pt] plot[domain=-2:3] (\x,{-1.5-0.5*cos((pi*(4+2*\x)/5) r)});		
		\draw[red, line width = 2pt] plot[domain=3:7] (\x,{-1-cos((pi*(11-\x)/4) r)});
		\draw[red, line width = 2pt] plot[domain=3:7] (\x,{-1.5-cos((pi*(11-\x)/4) r)});
		\draw[red, line width = 1pt, arrows=->] (2,0)--(2,-2.1);
		\node at (2.3,-0.6) {${\color{red} w}$};
		\node at (-5.5,-1.8) {${\color{red} \Omega_1(w)}$};
		\node at (-3,0.5) {${\color{red} \Omega_2(w)}$};
		\draw (-3.65,0.5) edge[->,bend right,line width = 1pt] (-5,-1.15);
		\node at (0,-3.5) {$D$};
		\node at (6,-1.75) {{\color{red} $\Sigma(w)$}};
		\draw (5.5,-1.75) edge[->,bend left, line width = 1pt] (4.7,-1.8);
		\node at (-7.8,1) {$z$};
		\draw[black, line width = 1pt, arrows = ->] (-7.5,-3)--(-7.5,1);
		\node at (-8,-2.5) {$-H$};
		\draw[black, line width = 1pt] (-7.6,-2.5)--(-7.4,-2.5);
		\node at (-7.8,-0.5) {$0$};
		\draw[black, line width = 1pt] (-7.6,-0.5)--(-7.4,-0.5);
		\node at (-7.8,0) {$d$};
		\draw[black, line width = 1pt] (-7.6,0)--(-7.4,0);
		\node at (-7,-3.25) {$-L$};
		\node at (7,-3.25) {$L$};
		\draw[black, line width = 1pt, arrows = <->] (-7,-2.75)--(7,-2.75);
		\node at (-4,-3.25) {${\color{red} \mathcal{C}(w)}$};
		\draw (-3.5,-3.25) edge[->,bend right, line width = 1pt] (3,-2.55);
		\draw (-3.5,-3.25) edge[->,bend right,line width = 1pt] (-2.25,-2.55);
		\draw[black, line width = 2pt] (-7,-2.5)--(7,-2.5);
	\end{tikzpicture}
	\caption{Geometry of $\Omega(w)$ for a state $w\in \mathcal{S}$ with non-empty and disconnected coincidence set.}\label{Fig3}
\end{figure}
%%%%%%%%%%%%%%%%

%%%%%%%%%%%%%%%%
\begin{proposition}\label{ACDC}
Let $\kappa>0$ and $v\in\bar{\mathcal{S}}$ be such that $\|v\|_{H^2(D)} \le \kappa$. 
\begin{itemize}
	\item [\textbf{(a)}] The unique minimizer $\psi_v\in \mathcal{A}(v)$ of $\mathcal{J}(v)$ on $\mathcal{A}(v)$ provided by Lemma~\ref{L1} satisfies
	\begin{equation*}
		\psi_{v,i} = \psi_v|_{\Omega_i(v)} \in H^2(\Omega_i(v))\,, \qquad i=1,2\,,
	\end{equation*}
and is a strong solution to the transmission problem \eqref{psiS}. Moreover, there is $c_1(\kappa)>0$ such that 
\begin{equation}\label{king3}
	\|\psi_v\|_{H^1(\Omega(v))} + \|\psi_{v,1}\|_{H^2(\Omega_1(v))} + \|\psi_{v,2}\|_{H^2(\Omega_2(v))}\le c_1(\kappa)\,.
\end{equation}
\item[\textbf{(b)}] Consider a sequence $(v_n)_{n\ge 1}$ in $\bar{\mathcal{S}}$ satisfying
\begin{equation}\label{vh2}
	\|v_n\|_{H^2(D)}\le \kappa\,, \quad n\ge 1\,, \;\;\text{ and }\;\; \lim_{n\to\infty} \|v_n-v\|_{H^1(D)} = 0 \,.
\end{equation}
If $i\in\{1,2\}$ and $U_i$ is an open subset of $\Omega_i(v)$ such that $\bar U_i$ is a compact subset of $\Omega_i(v)$, then
\begin{equation*}
	\psi_{v_n,i}\rightharpoonup \psi_{v,i} \quad\text{in}\quad H^2(U_i)\,,
\end{equation*}
recalling that $\psi_{v_n,i} = \psi_{v_n}|_{\Omega_i(v_n)}$.
\end{itemize}
\end{proposition}
%%%%%%%%%%%%%%%%

The proof involves three steps: we first establish Proposition~\ref{ACDC}~\textbf{(b)} under the additional assumption
\begin{equation*}
	\sup_{n\ge 1}\left\{ \|\psi_{v_n,1}\|_{H^2(\Omega_1(v_n))} + \|\psi_{v_n,2}\|_{H^2(\Omega_2(v_n))} \right\} < \infty\,.
\end{equation*}
Building upon this result, we take advantage of the density of $\mathcal{S}\cap W_\infty^2(D)$ in $\bar{\mathcal{S}}$ and of the estimates derived in Proposition~\ref{P2} to verify Proposition~\ref{ACDC}~\textbf{(a)} by a compactness argument. Combining the previous steps leads us finally to a complete proof of Proposition~\ref{ACDC}~\textbf{(b)}. We thus start with the proof of Proposition~\ref{ACDC}~\textbf{(b)} when the solutions $(\psi_{v_n})_{n\ge 1}$ to \eqref{psiS} associated with the sequence $(v_n)_{n\ge 1}$ satisfies the above additional bound. We state this result as a separate lemma for definiteness.

%%%%%%%%%%%%%%%%
\begin{lemma}\label{lez1}
	Let $\kappa>0$ and $v\in\bar{\mathcal{S}}$ be such that $\|v\|_{H^2(D)} \le \kappa$ and consider a sequence $(v_n)_{n\ge 1}$ in $\bar{\mathcal{S}}$ satisfying \eqref{vh2}. Assume further that, for each $n\ge 1$, $(\psi_{v_n,1},\psi_{v_n,2})$ belongs to $H^2(\Omega_1(v_n))\times H^2(\Omega_2(v_n))$ and that there is $\mu>0$ such that 
	\begin{equation}
		\|\psi_{v_n,1}\|_{H^2(\Omega_1(v_n))} + \|\psi_{v_n,2}\|_{H^2(\Omega_2(v_n))} \le \mu\,, \qquad n\ge 1\,. \label{uh2}
	\end{equation}
Then $\psi_{v,i} \in H^2(\Omega_i(v))$, $i=1,2$. In addition, if $i\in\{1,2\}$ and $U_i$ is an open subset of $\Omega_i(v)$ such that $\bar U_i$ is a compact subset of $\Omega_i(v)$, then
\begin{equation*}
	\psi_{v_n,i}\rightharpoonup \psi_{v,i} \quad\text{in}\quad H^2(U_i)
\end{equation*}
and
\begin{equation}
	\|\psi_{v,1}\|_{H^2(\Omega_1(v))} + \|\psi_{v,2}\|_{H^2(\Omega_2(v))} \le \mu\,. \label{z101}
\end{equation} 
\end{lemma}
%%%%%%%%%%%%%%%%

%\cgr{The proof is very close to that of \cite[Proposition~3.13 \& Corollary~3.14]{ARMA20}, so that we omit the details here and refer to \cite{ARMA20} instead.}
The proof of Lemma~\ref{lez1} is very close to that of \cite[Proposition~3.13 \& Corollary~3.14]{ARMA20}. For the sake of completeness we provide a detailed proof in Appendix~\ref{pr.acdc}.

\begin{proof}[Proof of Proposition~\ref{ACDC}~\textbf{(a)}]
Let $v\in \bar{\mathcal{S}}$ be such that $\|v\|_{H^2(D)}\le \kappa$. We may choose a sequence $(v_n)_{n\ge 1}$ in $\mathcal{S}\cap W_\infty^2(D)$ satisfying 
\begin{equation}
v_n\rightarrow v \ \text{ in }\ H^2(D)\,,\qquad \sup_{n\ge 1}\,\|v_n\|_{H^2(D)}\le 2\kappa\,. \label{z1103}
\end{equation}
Owing to \eqref{z1103} and the regularity property $v_n\in \mathcal{S}\cap W_\infty^2(D)$, $n\ge 1$, Proposition~\ref{C3} guarantees that $(\psi_{v_n,1},\psi_{v_n,2})$ belongs to $H^2(\Omega_1(v_n))\times H^2(\Omega_2(v_n))$ and $(\psi_{v_n})_{n\ge 1}$ satisfies \eqref{uh2} with $\mu=c_0(2\kappa)$. We then infer from Lemma~\ref{lez1} that $(\psi_{v,1},\psi_{v,2})$ belongs to $H^2(\Omega_1(v))\times H^2(\Omega_2(v))$ and satisfies
\begin{equation*}
	\|\psi_{v,1}\|_{H^2(\Omega_1(v))} + \|\psi_{v,2}\|_{H^2(\Omega_2(v))} \le c_0(2\kappa)\,. 
\end{equation*}
Combining the above bound with \eqref{204} and Lemma~\ref{lez1} gives \eqref{king3}. It remains to check that $\psi_v$ is a strong solution to \eqref{psiS} which can be  done as in \cite[Corollary~3.14]{ARMA20}, see Appendix~\ref{pr.acdc} for details. %\cgr{Checking that $\psi_v$ is a strong solution to \eqref{psiS} is then done as in \cite[Corollary~3.14]{ARMA20}, to which we refer. }
\end{proof}

\begin{proof}[Proof of Proposition~\ref{ACDC}~\textbf{(b)}]
Proposition~\ref{ACDC}~\textbf{(b)} is now a straightforward consequence of Proposition~\ref{ACDC}~\textbf{(a)} and Lemma~\ref{lez1}.
\end{proof}

\begin{proof}[Proof of Theorem~\ref{Thm1}]
The proof of Theorem~\ref{Thm1} readily follows from  Proposition~\ref{ACDC}~{\bf (a)}.
\end{proof}

We supplement the $H^2$-weak continuity of $\psi_v$ with respect to $v$ reported in Proposition~\ref{ACDC} with the continuity of the traces of $\nabla\psi_{v,2}$ on the upper and lower boundaries of $\Omega_2(v)$.

%%%%%%%%%%%%%%%%
\begin{proposition}\label{U2}
Let $\kappa>0$ and $v\in\bar{\mathcal{S}}$ be such that $\|v\|_{H^2}(D)\le \kappa$ and consider a sequence $(v_n)_{n\ge 1}$ in~$\bar{\mathcal{S}}$ satisfying \eqref{vh2}. Then, for  $p\in [1,\infty)$,
\begin{align}
\nabla\psi_{v_n,2}(\cdot,v_n) & \rightarrow \nabla\psi_{v,2}(\cdot,v) \quad \text{ in }\quad L_p(D,\R^2)\,, \label{x}\\
\nabla\psi_{v_n,2}(\cdot,v_n+d) & \rightarrow \nabla\psi_{v,2}(\cdot,v+d) \quad \text{ in }\quad L_p(D,\R^2)\,, \label{x1}
\end{align}
and
\begin{equation}\label{xx}
\|\nabla\psi_{v,2}(\cdot,v) \|_{L_p(D,\R^2)} + \|\nabla\psi_{v,2}(\cdot,v+d) \|_{L_p(D,\R^2)} \le c(p,\kappa)\,.
\end{equation}
\end{proposition}
%%%%%%%%%%%%%%%%

\begin{proof}
Recall first from \eqref{king3} that
\begin{equation}\label{king3a}
\|\psi_{v_n,2}\|_{H^2(\Omega_2(v_n))}   \le c_1(\kappa)\,,\qquad n\ge 1\,.
\end{equation}
As in the proof of Lemma~\ref{L3} we map $\Omega_2(v)$ onto the rectangle $\mathcal{R}_2 =D\times (1,1+d)$ and define, for $(x,\eta)\in \mathcal{R}_2$ and $n\ge 1$,
 \begin{equation*}
 	\phi_n(x,\eta)  := \psi_{v_n,2}(x,\eta+v_n(x)-1)\,,\qquad \phi(x,\eta) := \psi_{v,2}(x,\eta+v(x)-1)\,.
 \end{equation*}
Let $q\in (1,2)$. Since
\begin{align*}
	\nabla\phi_n(x,\eta) & = \Big( \partial_x \psi_{v_n} + \partial_x v_n \partial_z \psi_{v_n} \,,\, \partial_z \psi_{v_n} \Big)(x,\eta+v_n(x)-1)\,, \\
	\partial_x^2 \phi_n(x,\eta) & = \Big( \partial_x^2 \psi_{v_n} + 2\partial_x v_n \partial_x \partial_z \psi_{v_n} + (\partial_x v_n)^2 \partial_z^2 \psi_{v_n} + \partial_x^2 v_n \partial_z \psi_{v_n} \Big)(x,\eta+v_n(x)-1)\,, \\
	\partial_x\partial_\eta\phi_n(x,\eta) & = \Big( \partial_x\partial_z \psi_{v_n} + \partial_x v_n \partial_z^2 \psi_{v_n}\Big)(x,\eta+v_n(x)-1)\,, \\
	\partial_\eta^2 \phi_n(x,\eta) & = \partial_z^2\psi_{v_n}(x,\eta+v_n(x)-1)\,, 
\end{align*}
it follows from \eqref{vh2}, \eqref{king3a}, the continuous embedding of $H^2(D)$ in $C^1(\bar{D})$, and that of $H^1(\mathcal{R}_2)$ in $L^{2q/(2-q)}(\mathcal{R}_2)$ that
\begin{equation}\label{q1}
\phi_n\in W_q^2(\mathcal{R}_2)\;\;\text{ with }\;\; \| \phi_n\|_{W_q^2(\mathcal{R}_2)}\le c(q,\kappa)\,,\qquad n\ge 1\,.
\end{equation}

Now, given $p\in [1,\infty)$, we choose $q\in (1,\min\{2,p\})$ satisfying $1<2/q<1 + 1/p$ and $s\in (2/q-1/p,1)$. Since 
$$
\phi_n \rightharpoonup \phi \quad \text{ in }\quad W_q^2(\mathcal{R}_2)
$$
by \eqref{204},  \eqref{q1}, and Proposition~\ref{ACDC}, the continuity of the trace as a mapping from $W_q^1(\mathcal{R}_2)$ to $W_q^{1-1/q}(D\times\{1\})$ and the compactness of the embedding of $W_q^{1-1/q}(D)$ in $L_p(D)$ imply that 
\begin{equation}
\nabla\phi_n (\cdot,1) \rightarrow \nabla\phi (\cdot,1) \quad \text{ in }\quad W_q^{s-1/q}(D)\label{z107}
\end{equation}
and
\begin{equation}
\|\nabla\phi(\cdot,1)\|_{L_p(D)} \le c(p,\kappa)\,. \label{z108}
\end{equation}
That is,
$$
\partial_z\psi_{v_n,2}(\cdot,v_n)= \partial_\eta\phi_n (\cdot,1) \rightarrow \partial_\eta\phi (\cdot,1)=\partial_z\psi_{v,2}(\cdot,v) \quad \text{ in }\quad L_p(D)
$$
and, recalling \eqref{vh2} and the continuous embedding of $H^2(D)$ in $C^1(\bar{D})$,
\begin{equation*}
\begin{split}
\partial_x\psi_{v_n,2}(\cdot,v_n)= \partial_x \phi_n (\cdot,1)&-\partial_x v_n\partial_\eta\phi_n (\cdot,1)\\
& \rightarrow \partial_x \phi (\cdot,1)-\partial_x v\partial_\eta\phi (\cdot,1)
=\partial_x\psi_{v,2}(\cdot,v) \quad \text{ in }\quad L_p(D)\,.
\end{split}
\end{equation*}
Furthermore, \eqref{z107} and \eqref{z108}, along with the bound $\|v\|_{H^2(D)}\le \kappa$ and the continuous embedding of $H^2(D)$ in $C^1(\bar{D})$, entail that
$$
\|\nabla\psi_{v,2}(\cdot,v)\|_{L_p(D)} \le c(p,\kappa)\,,
$$
which proves \eqref{x} and the first bound in \eqref{xx}. Clearly, \eqref{x1} and the second bound in \eqref{xx} are shown in the same way.
\end{proof}

\begin{proof}[Proof of Theorem~\ref{Thm2}]
The proof of Theorem~\ref{Thm2} is now a consequence of Proposition~\ref{C3} for \eqref{y1}, Proposition~\ref{ACDC}~{\bf (b)} for \eqref{y2}, and Proposition~\ref{U2} for \eqref{y3}.
\end{proof}

%%%%%%%%%%%%%%%%
%%%%%%%%%%%%%%%%
\appendix
\section{The Identity \eqref{DI}}\label{sec.id}
%%%%%%%%%%%%%%%%
%%%%%%%%%%%%%%%%

This appendix is devoted to the proof of the identity~\eqref{DI}, which can be seen as a variant of \cite[Lemma~4.3.1.2]{Gr85} with piecewise constants linear constraints on the boundaries instead of constant ones.

%%%%%%%%%%%%%%%%
\begin{lemma}\label{ID}
Let $\mathcal{R} = D \times (0,1+d)$ and consider $(V,W)\in H^1(\mathcal{R},\mathbb{R}^2)$ satisfying
\begin{subequations}\label{app1}
\begin{align}
	V(x,0) = V(x,1+d) & = 0\,, \qquad x\in D=(-L,L)\,, \label{app1a}\\
	W(\pm L, \eta) + \tau^\pm(\eta) V(\pm L,\eta) & = 0\,, \qquad \eta\in (0,1+d)\,, \label{app1b}
\end{align}	
\end{subequations}
where $\tau^{\pm}$ are piecewise constants functions of the form
\begin{equation}
	\tau^\pm = \tau_1^\pm \mathbf{1}_{(0,1)} + \tau_2^\pm \mathbf{1}_{(1,1+d)} \label{app2}
\end{equation}
with $(\tau_1^+,\tau_1^-,\tau_2^+,\tau_2^-)\in \mathbb{R}^4$, featuring possibly a jump discontinuity at $\eta=1$. Then
\begin{equation*}
	\int_\mathcal{R}\partial_x V \partial_\eta W \,\rd (x,\eta) = \int_\mathcal{R} \partial_{\eta}V \partial_x W\,\rd (x,\eta)\,.
\end{equation*}
\end{lemma}
%%%%%%%%%%%%%%%%

When $\tau_1^\pm = \tau_2^\pm= 0$, Lemma~\ref{ID} is a straightforward consequence of \cite[Lemma~4.3.1.2]{Gr85}. The novelty here is the possibility of handling the jump discontinuity in \eqref{app1b} when $\tau_1^\pm \ne \tau_2^\pm$ in \eqref{app2}.

\medskip

The proof follows the lines of that of \cite[Lemma~4.3.1.2]{Gr85}. For $s\ge 1$, we introduce the space
\begin{equation*}
	\mathcal{G}^s(\mathcal{R}) := \{ (V,W)\in H^s(\mathcal{R},\mathbb{R}^2)\,:\, (V,W) \;\text{ satisfies } \eqref{app1}\}\,,
\end{equation*}
and first report the density of $\mathcal{G}^2(\mathcal{R})$ in $\mathcal{G}^1(\mathcal{R})$.

%%%%%%%%%%%%%%%%
\begin{lemma}\label{app10}
	$\mathcal{G}^2(\mathcal{R})$ is dense in $\mathcal{G}^1(\mathcal{R})$.
\end{lemma}
%%%%%%%%%%%%%%%%

 As in the proof of \cite[Lemma~4.3.1.3]{Gr85}, the core of the proof of Lemma~\ref{app10} is to establish the density of the space $\mathcal{Z}^2(\partial\mathcal{R})$ of traces of functions in $\mathcal{G}^2(\mathcal{R})$ in the space $\mathcal{Z}^1(\partial\mathcal{R})$ of traces of functions in $\mathcal{G}^1(\mathcal{R})$, after identifying these two trace spaces.  The proof is almost identical to that of \cite[Lemma~4.3.1.3]{Gr85} and we postpone it to the end of this appendix.  %\cgr{Since the proof is almost identical to that of \cite[Lemma~4.3.1.3]{Gr85}, we omit it here.} 

\begin{proof}[Proof of Lemma~\ref{app1}] Due to Lemma~\ref{app10}, it suffices to prove the identity in Lemma~\ref{app1} when $(V,W)$ belongs to $\mathcal{G}^2(\mathcal{R})$. This additional regularity allows us to use integration by parts to interchange the derivatives and guarantees the continuity of both $V$ and $W$ on $\bar{\mathcal{R}}$. Indeed,  $H^2(\mathcal{R})$ embeds continuously in $C^\alpha(\bar{\mathcal{R}})$ for all $\alpha\in (0,1)$ by \cite[Chapter~2, Theorem~3.8]{Necas2012} and we deduce that
\begin{equation}
	(V,W)\in C(\bar{\mathcal{R}},\mathbb{R}^2)\,. \label{app3}
\end{equation}
Next, after integrating by parts, 
\begin{align*}
	J(V,W) := & \, \int_\mathcal{R} \big( \partial_x V \partial_\eta W - \partial_\eta V \partial_x W \big)\ \mathrm{d}(x,\eta) \\
	= & \, \int_0^{1+d} \Big[ (V \partial_\eta W)(x,\eta) \Big]_{x=-L}^{x=L}\ \rd\eta - \int_\mathcal{R} V \partial_x\partial_\eta W\ \rd (x,\eta) \\
	& \qquad - \int_D \Big[ (V\partial_x W)(x,\eta) \Big]_{\eta=0}^{\eta=1+d} + \int_\mathcal{R} V \partial_x\partial_\eta W\ \rd (x,\eta)\,.
\end{align*}
Since $V(x,0) = V(x,1+d) = 0$ for $x\in D$ by \eqref{app1a} and the second and fourth terms cancel each other out, we obtain 
\begin{equation*}
	J(V,W) = \int_0^{1+d} V(L,\eta) \partial_\eta W(L,\eta)\ \rd\eta - \int_0^{1+d} V(-L,\eta) \partial_\eta W(-L,\eta)\ \rd\eta\,.
\end{equation*}
Now, according to \eqref{app1b} and the regularity of $V$ and $W$,
\begin{align*}
	\partial_\eta W(\pm L,\eta) & = - \tau_1^\pm \partial_\eta V(\pm L,\eta)\,, \qquad \eta\in (0,1)\,, \\
	\partial_\eta W(\pm L,\eta) & = - \tau_2^\pm \partial_\eta V(\pm L,\eta)\,, \qquad \eta\in (1,1+d)\,,
\end{align*} 
so that, since $[\eta\mapsto V(\pm L,\eta)] \in C([0,1+d])$ by \eqref{app3},
\begin{align}
	J(V,W) & = - \tau_1^+ \int_0^1 (V \partial_\eta V)(L,\eta)\ \rd \eta - \tau_2^+ \int_1^{1+d} (V \partial_\eta V)(L,\eta)\ \rd \eta \nonumber \\
	& \qquad + \tau_1^- \int_0^1 (V \partial_\eta V)(-L,\eta)\ \rd \eta + \tau_2^- \int_1^{1+d} (V \partial_\eta V)(-L,\eta)\ \rd \eta \nonumber \\
	& = - \tau_1^+ \frac{V(L,1)^2 - V(L,0)^2}{2} - \tau_2^+ \frac{V(L,1+d)^2 - V(L,1)^2}{2} \nonumber \\
	& \qquad + \tau_1^- \frac{V(-L,1)^2 - V(-L,0)^2}{2} + \tau_2^- \frac{V(-L,1+d)^2 - V(-L,1)^2}{2} \nonumber \\
	& = \frac{\tau_1^+}{2} V(L,0)^2 - \frac{\tau_1^-}{2} V(-L,0)^2 - \frac{\tau_2^+}{2} V(L,1+d)^2 + \frac{\tau_2^-}{2} V(-L,1+d)^2 \label{app4}\\
	& \qquad - \frac{\tau_1^+ - \tau_2^+}{2} V(L,1)^2 + \frac{\tau_1^- - \tau_2^-}{2} V(-L,1)^2\,. \nonumber
\end{align}
On the one hand, it follows from \eqref{app1} and the continuity \eqref{app3} of $V$ that
\begin{equation}
\begin{split}
	V(\pm L,0) & = \lim_{x\to \pm L} V(x,0) = 0\,, \\
	V(\pm L,1+d) & = \lim_{x\to \pm L} V(x,1+d) = 0\,.
\end{split} \label{app5}
\end{equation}
On the other hand, using \eqref{app1b} along with the continuity \eqref{app3} gives
\begin{align*}
	\tau_1^\pm V(\pm L, 1) & = \lim_{\eta\nearrow 1} \tau^\pm(\eta) V(\pm L,\eta)  = - \lim_{\eta\nearrow 1} W(\pm L, \eta) \\
	& = - W(\pm L,1) = - \lim_{\eta\searrow 1} W(\pm L, \eta) = \lim_{\eta\searrow 1} \tau^\pm(\eta) V(\pm L,\eta) \\
	& =  \tau_2^\pm V(\pm L, 1)\,. 
\end{align*}
Consequently,
\begin{equation}
	\left( \tau_1^\pm - \tau_2^\pm \right) V(\pm L, 1)= 0\,. \label{app6}
\end{equation}
Combining \eqref{app4}, \eqref{app5}, and \eqref{app6} leads us to $J(V,W)=0$ and we have proved that
\begin{equation}
	J(V,W) = 0\,, \qquad (V,W)\in \mathcal{G}^2(\mathcal{R})\,. \label{app7}
\end{equation}
In other words, the identity stated in Lemma~\ref{ID} is valid for $(V,W)\in \mathcal{G}^2(\mathcal{R})$. 
\end{proof}

We provide here a proof of the density of $\mathcal{G}^2(\mathcal{R})$ in $\mathcal{G}^1(\mathcal{R})$ as claimed in Lemma~\ref{app10}. It is adapted from that of \cite[Lemma~4.3.1.3]{Gr85}.

\begin{proof}[Proof of Lemma~\ref{app10}] 
To cast the problem under consideration in a form which is as close as possible to that used in \cite[Section~4.3.1]{Gr85}, we recall that $\mathcal{R}$ is a polygon with vertices
\begin{equation*}
	S_1 := (L,0)\,, \qquad S_2 := (L,1+d)\,, \qquad S_3 := (-L,1+d)\,, \qquad S_4 := (-L,0)\,, 
\end{equation*}
and edges
\begin{equation*}
	\begin{split}	
		& \Gamma_1 := D \times \{0\} = (S_4,S_1)\,, \qquad \Gamma_2 := \{L\}\times (0,1+d) = (S_1,S_2)\,,  \\
		& \Gamma_3 := D \times \{1+d\} = (S_2,S_3)\,, \qquad \Gamma_4 := \{-L\}\times (0,1+d) = (S_3,S_4)\,,
	\end{split}
\end{equation*}
%%%%%%%%%%%%%%%%
\begin{figure}
	\begin{tikzpicture}[scale=0.7]
		\draw[black, line width = 1.5pt] (-6,0)--(6,0);
		\draw[black, line width = 2pt] (-6,0)--(-6,-5);
		\draw[black, line width = 2pt] (6,-5)--(6,0);
		\draw[black, line width = 2pt] (-6,-5)--(6,-5);
		\node at (-6,-5.5) {$-L$};
		\node at (6,-5.5) {$L$};
		\node at (0.5,-2.5) {$\mathcal{R}$};
		\node at (7.25,-5.25) {$x$};
		\draw[black, dashed, line width = 1pt, arrows = ->] (-7,-5.0)--(7,-5.0);
		\node at (-0.3,1) {$\eta$};
		\draw[black, dashed, line width = 1pt, arrows = ->] (0,-6)--(0,1);
		\node at (-6.5,-4.5) {$S_4$};
		\node at (6.5,-4.5) {$S_1$};
		\node at (6.5,0) {$S_2$};
		\node at (-6.5,0) {$S_3$};
		\node at (-5.5,-2.5) {$\Gamma_4$};
		\node at (-0.5,-0.5) {$\Gamma_3$};
		\node at (5.5,-2.5) {$\Gamma_2$};
		\node at (-0.5,-4.5) {$\Gamma_1$};
		 \node at (0.35,-5.35) {$0$};
        \node at (0.35,-0.78) {$1$};
        \draw[black,line width = 1pt] (-0.1,-0.78)--(0.1,-0.78);
        \node at (0.75,0.35) {$1+d$}; 
	\end{tikzpicture}
	\caption{The rectangle $\mathcal{R}$.}\label{Fig4}
\end{figure}
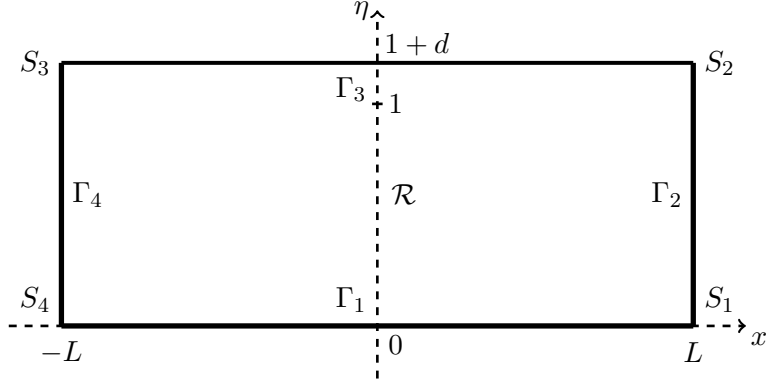
%%%%%%%%%%%%%%%%
see Figure~\ref{Fig4}. We next introduce local charts $(Y_i)_{1\le i \le 4}$ in the neighborhood of the vertices $(S_i)_{1\le i \le 4}$ defined by
\begin{align*}
	Y_1(s) & := \left\{
	\begin{array}{lc}
		(L,s)\,, & s\in (0,1+d)\,, \\
		(L+s,0)\,, & s\in (-2L,0)\,, 
	\end{array}\right. \\
	Y_2(s) & := \left\{
	\begin{array}{lc}
		(L-s,1+d)\,, & s\in (0,2L)\,, \\
		(L,1+d+s)\,, & s\in (-1-d,0)\,, 
	\end{array}\right. \\
	Y_3(s) & := \left\{
	\begin{array}{lc}
		(-L,1+d-s)\,, & s\in (0,1+d)\,, \\
		(-L-s,1+d)\,, & s\in (-2L,0)\,, 
	\end{array}\right. \\
	Y_4(s) & := \left\{
	\begin{array}{lc}
		(-L+s,0)\,, & s\in (0,2L)\,, \\
		(-L,-s)\,, & s\in (-1-d,0)\,. 
	\end{array}\right.
\end{align*}
We also set 
\begin{equation}
	(\lambda_1,\lambda_2,\lambda_3,\lambda_4) := (1,\tau^+,1,\tau^-) \,,\qquad (\mu_1,\mu_2,\mu_3,\mu_4) := (0,1,0,1)\,, \label{S20}
\end{equation}
and observe that the main difference to the situation studied in \cite[Section~4.3.1]{Gr85} is that $\lambda_2$ and $\lambda_4$ are not constants, but piecewise constant and possibly discontinuous functions. 

Now, according to \cite[Theorem~1.5.2.3]{Gr85} and the definition of $\mathcal{G}^1(\mathcal{R})$, the space $\mathcal{Z}^1(\partial\mathcal{R})$ is the subspace of $\bigotimes_{i=1}^4 H^{1/2}(\Gamma_i,\mathbb{R}^2)$ defined by
\begin{equation}
	\begin{split}
		& (v_i,w_i)\in H^{1/2}(\Gamma_i,\mathbb{R}^2)\,, \qquad 1\le i \le 4\,, \\
		& \lambda_i v_i + \mu_i w_i = 0 \;\text{ on }\; \Gamma_i\,, \qquad 1\le i \le 4\,, \\
		& \int_0^\delta \frac{\big[ v_{i+1}(Y_i(s)) - v_i(Y_i(-s)) \big]^2}{s}\ \rd s < \infty\,, \qquad 1\le i \le 4\,, \\
		& \int_0^\delta \frac{\big[ w_{i+1}(Y_i(s)) - w_i(Y_i(-s)) \big]^2}{s}\ \rd s < \infty\,, \qquad 1\le i \le 4\,,
	\end{split}\label{S21}
\end{equation}
with $\delta := \min\{1+d,2L\}$ and $(v_5,w_5) := (v_1,w_1)$. Owing to \eqref{S20} and \eqref{S21}, the integrability properties listed in \eqref{S21} simplify  and $\mathcal{Z}^1(\partial\mathcal{R})$ is isomorphic to $\mathcal{P}^1(\partial\mathcal{R},\mathbb{R}^2)$ with
\begin{equation*}
	\mathcal{P}^1(\partial\mathcal{R},\mathbb{R}^2) := \mathcal{P}^1(\partial\mathcal{R})\times \mathcal{P}^1(\partial\mathcal{R})\,,
\end{equation*}
where $\mathcal{P}^1(\partial\mathcal{R})$ is the subspace of $\bigotimes_{i=1}^4 H^{1/2}(\Gamma_i)$ defined by
\begin{equation}
	\begin{split}
		& \varphi_i \in H^{1/2}(\Gamma_i)\,,  \qquad 1\le i \le 4\,, \\ 
		& \int_0^\delta \left(  \frac{|\varphi_i(Y_i(-s))|^2}{s} + \frac{|\varphi_{i+1}(Y_{i}(s))|^2}{s} \right)\ \rd s < \infty\,,  \qquad 1\le i \le 4\,,
	\end{split}\label{S21b}
\end{equation}
with $\varphi_5 := \varphi_1$. 

We next turn to $\mathcal{Z}^2(\partial\mathcal{R})$ and deduce from \cite[Theorem~1.5.2.8]{Gr85} that it is the subspace of $\bigotimes_{i=1}^4 H^{3/2}(\Gamma_i,\mathbb{R}^2)$ defined by
\begin{equation}
	\begin{split}
		& (v_i,w_i)\in H^{3/2}(\Gamma_i,\mathbb{R}^2)\,, \qquad 1\le i \le 4\,, \\
		& \lambda_i v_i + \mu_i w_i = 0 \;\text{ on }\; \Gamma_i\,, \qquad 1\le i \le 4\,, \\
		& (v_{i+1},w_{i+1})(S_i) = (v_i,w_i)(S_i)\,, \qquad 1\le i \le 4\,,
	\end{split}\label{S22}
\end{equation}
with $(v_5,w_5)=(v_1,w_1)$. As above, the continuity requirements in \eqref{S22} simplify due to \eqref{S20} and we conclude that $\mathcal{Z}^2(\partial\mathcal{R})$ is isomorphic to $\mathcal{P}^2(\partial\mathcal{R},\mathbb{R}^2)$ with
\begin{equation*}
	\mathcal{P}^2(\partial\mathcal{R},\mathbb{R}^2) := \mathcal{P}^2(\partial\mathcal{R})\times \mathcal{P}^2(\partial\mathcal{R})\,,
\end{equation*}
where $\mathcal{P}^2(\partial\mathcal{R})$ is the subspace of $\bigotimes_{i=1}^4 H^{3/2}(\Gamma_i)$ defined by
\begin{equation}
	\begin{split}
		& \varphi_i \in H^{3/2}(\Gamma_i)\,,  \qquad 1\le i \le 4\,, \\ 
		& \varphi_i(S_i) = \varphi_{i+1}(S_i) = 0\,,  \qquad 1\le i \le 4\,,
	\end{split}\label{S22b}
\end{equation}
with $\varphi_5 = \varphi_1$. 

Now, since $\bigotimes_{i=1}^4 C_c^\infty(\Gamma_i)$ is dense in $\mathcal{P}^1(\partial\mathcal{R})$ by Lemma~\ref{lem.d} below and obviously included in $\mathcal{P}^2(\partial\mathcal{R})$, the density of $\mathcal{P}^2(\partial\mathcal{R})$ in $\mathcal{P}^1(\partial\mathcal{R})$ follows, and that of $\mathcal{Z}^2(\partial\mathcal{R})$ in $\mathcal{Z}^1(\partial\mathcal{R})$ as well. The remainder of the proof is then the same as in \cite[Lemma~4.3.1.3]{Gr85}, to which we refer.
\end{proof}

%%%%%%%%%%%%%%%%
%%%%%%%%%%%%%%%%
\section{A Density Result}\label{sec.dp}
%%%%%%%%%%%%%%%%
%%%%%%%%%%%%%%%%

In this appendix, we recall a density result which is stated without proof in \cite[Lemma~4.3.1.3]{Gr85} and used in the proof of Lemma~\ref{app10}. We provide a proof for the sake of completeness.
 
%%%%%%%%%%%%%%%%
\begin{lemma}\label{lem.d}
	The space $C_c^\infty((0,1))$ is dense in $H^{1/2}(0,1)\cap L_2((0,1),\rd x/x)$.
\end{lemma}
%%%%%%%%%%%%%%%%

A preliminary step is the analogue of Lemma~\ref{lem.d} when $(0,1)$ is replaced by $(0,\infty)$.

%%%%%%%%%%%%%%%
\begin{lemma}\label{lem.d1}
	The space $C_c^\infty((0,\infty))$ is dense in $H^{1/2}(0,\infty)\cap L_2((0,\infty),\rd x/x)$.
\end{lemma}
%%%%%%%%%%%%%%%%

\begin{proof}
Set $X:= H^{1/2}(0,\infty)\cap L_2((0,\infty),\rd x/x)$. The proof is divided into two steps: we first show that  $C_c^\infty([0,\infty))\cap X$ is dense in $X$ with an argument from \cite{MS1964}. We then use a method of truncation as in \cite{Tr1978} to complete  the proof.
	
\smallskip
	
\noindent\textit{Step~1: Density of $C_c^\infty([0,\infty))\cap X$ in $X$.} Consider $f\in X$ and $\delta\in (0,1)$. As in \cite{MS1964}, we define $I_j := (1/j,j)$ for $j\ge 1$ and note that
\begin{equation*}
	D_j := I_{j+1}\setminus \bar{I}_{j-1} = \left( \frac{1}{j+1} , \frac{1}{j-1} \right) \cup (j-1,j+1)\,, \qquad j\ge 2\,.
\end{equation*}
Since $(0,\infty) = I_2 \cup \left( \bigcup_{j\ge 2} D_j \right)$, there exists a partition of unity $(\psi_j)_{j\ge 1}$ consisting of non-negative functions in $C_c^\infty((0,\infty))$ such that
\begin{subequations}\label{pu}
\begin{equation}
	\mathrm{supp}\,\psi_1 \subset I_2 \;\;\text{ and }\;\; \mathrm{supp}\,\psi_j \subset D_j \;\text{ for }\;  j\ge 2\,, \label{pu1}
\end{equation} 
\begin{equation}
	\sum_{j=1}^\infty \psi_j(x) = 1 \;\text{ for }\;  x\in (0,\infty)\,, \label{pu2}
\end{equation}
and, for every compact subset $K$ of $(0,\infty)$, there exist an integer $\ell_K\ge 1$ and an open subset $\mathcal{O}_K$ of $(0,\infty)$ such that 
\begin{equation}
	K\subset \mathcal{O}_K \;\;\text{ and }\;\; \sum_{j=1}^{\ell_K} \psi_j(x) = 1 \;\text{ for }\; x\in \mathcal{O}_K\,.
	\label{pu3}
\end{equation}
\end{subequations}
Observe that, if $x\in (1/2,1)\cup (1,2)$, then $x\in I_2\cap D_2$ and $x\not\in D_k$ for $k\ge 3$, while, if $x\in (1/(j+1),1/j)\cup (j,j+1)$ for some $j\ge 2$, then $x\in D_j\cap D_{j+1}$ and $x\not\in D_k$ for $1 \le k\le j-1$ and $k\ge j+2$. In addition, $1\in I_2$ but $1\not\in D_k$ for $k\ge 2$. Consequently, given $x\in (0,\infty)$, the series in \eqref{pu2} has at most two non-vanishing terms. 

Next, let $(\varrho_\varepsilon)_{\varepsilon\in (0,1)}$ be a family of $C^\infty$-smooth mollifiers satisfying
\begin{equation}
	\mathrm{supp}\, \varrho_\varepsilon \subset (-\varepsilon,\varepsilon) \;\;\text{ and }\;\; \int_{\mathbb{R}} \varrho_\varepsilon(x)\ \mathrm{d}x = 1 \;\text{ for }\; \varepsilon\in (0,1)\,. \label{pu4}
\end{equation}  
Let $j\ge 3$ and $\varepsilon\in (0,1)$. Owing to the properties of the convolution, 
\begin{equation*}
	\mathrm{supp}\, \big(\varrho_\varepsilon * (\psi_j f)\big) \subset \left( \frac{1}{j+1} - \varepsilon , \frac{1}{j-1} + \varepsilon \right) \cup (j-1-\varepsilon , j+1+\varepsilon)\,,
\end{equation*}
so that, if $\varepsilon\in (0,1/(j+1)(j+2))$, then
\begin{subequations}\label{pu5}
	\begin{equation}
		\mathrm{supp}\, \big(\varrho_\varepsilon * (\psi_j f)\big) \subset \left( \frac{1}{j+2} , j+2 \right)\,. \label{pu5a}
	\end{equation}
Similarly, for $\varepsilon\in (0,1/12)$, 
	\begin{equation}
	\mathrm{supp}\, \big(\varrho_\varepsilon * (\psi_2 f)\big) \subset \left( \frac{1}{4} , 4 \right)\,, \label{pu5b}
\end{equation}
and, for $\varepsilon\in (0,1/6)$, 
\begin{equation}
	\mathrm{supp}\, \big(\varrho_\varepsilon * (\psi_1 f)\big) \subset \left( \frac{1}{3} , 3 \right)\,. \label{pu5c}
\end{equation}
\end{subequations}
Since $\big( \varrho_\varepsilon * (\psi_j f) \big)_{\varepsilon\in (0,1)}$ converges to $\psi_j f$ in $H^{1/2}(0,\infty)$ and in $L_2(0,\infty)$ for each $j\ge 1$ as $\varepsilon\to 0$, we may pick $\varepsilon_j\in (0,1)$ such that
\begin{subequations}\label{pu678}
	\begin{align}
		& \varepsilon_j \in \left( 0 , \frac{1}{(j+1)(j+2)} \right)\,, \qquad j\ge 1\,, \label{pu6}\\
		& \left\| \varrho_{\varepsilon_j} * (\psi_j f) - \psi_j f \right\|_{H^{1/2}(0,\infty)} \le \delta 2^{-j}\,, \qquad j\ge 1\,, \label{pu7} \\
		& \left\| \varrho_{\varepsilon_j} * (\psi_j f) - \psi_j f \right\|_{L_2(0,\infty)} \le \frac{\delta}{\sqrt{j+2}} 2^{-j}\,, \qquad j\ge 1\,. \label{pu8}
	\end{align}
\end{subequations}
A first consequence of \eqref{pu5}, \eqref{pu6}, and \eqref{pu8} is that, for $j\ge 1$, 
\begin{align}
	& \left\| \varrho_{\varepsilon_j} * (\psi_j f) - \psi_j f \right\|_{L_2((0,\infty),\mathrm{d}x/x)} \nonumber\\
	& \qquad = \left( \int_{1/(j+2)}^{j+2} \left| \left( \varrho_{\varepsilon_j} * (\psi_j f) - \psi_j f \right)(x) \right|^2\ \frac{\rd x}{x} \right)^{1/2} \nonumber \\
	& \qquad \le \sqrt{j+2} \left( \int_{1/(j+2)}^{j+2} \left| \left( \varrho_{\varepsilon_j} * (\psi_j f) - \psi_j f \right)(x) \right|^2\ \rd x \right)^{1/2} \nonumber \\
	& \qquad \le \sqrt{j+2} \left\| \varrho_{\varepsilon_j} * (\psi_j f) - \psi_j f \right\|_{L_2(0,\infty)} \nonumber \\
	& \qquad \le \delta 2^{-j}\,. \label{pu9}
\end{align}

Still following the argument in \cite{MS1964}, we now define
\begin{equation}
	F := \sum_{j=1}^\infty \varrho_{\varepsilon_j} * (\psi_j f)  \;\;\text{ and }\;\; F_k := \sum_{j=1}^k \varrho_{\varepsilon_j} * (\psi_j f) \,, \quad k\ge 1\,. \label{pu10}
\end{equation}
By \eqref{pu5}, given $x\in (0,\infty)$, 
\begin{equation*}
	\big( \varrho_{\varepsilon_j} * (\psi_j f) \big)(x) = 0 \;\text{ for }\; j \ge 2 + \max\left\{ x , \frac{1}{x}\right\}\,,
\end{equation*}
so that $F(x)$ is actually a finite sum. As a consequence, $F\in C^\infty([0,\infty))$ and, for $k\ge 2$ and $x\in I_k$, we infer from \eqref{pu} and \eqref{pu10} that
\begin{subequations}\label{pu11}
	\begin{equation}
		F(x) = \sum_{j=1}^{k+2} \big( \varrho_{\varepsilon_j} * (\psi_j f) \big)(x) = F_{k+2}(x) \label{pu11a}
	\end{equation}
and
	\begin{equation}
	f(x) = \sum_{j=1}^\infty (\psi_j f)(x) = \sum_{j=1}^{k+2} (\psi_j f)(x)\,. \label{pu11b}
\end{equation}
\end{subequations}

Let $k\ge 2$. Owing to \eqref{pu7} and \eqref{pu11},
\begin{align}
	\|f-F\|_{H^{1/2}(I_k)} & = \left\| \sum_{j=1}^{k+2} \varrho_{\varepsilon_j} * (\psi_j f) - \sum_{j=1}^{k+2} \psi_j f \right\|_{H^{1/2}(I_k)} \nonumber \\
	& \le \sum_{j=1}^{k+2} \| \varrho_{\varepsilon_j} * (\psi_j f) - \psi_j f \|_{H^{1/2}(I_k)} \le  \delta\,. \label{pu12}
\end{align}
In particular,
\begin{equation}\label{pu30}
	\|F\|_{H^{1/2}(I_k)} \le \delta + \|f\|_{H^{1/2}(I_k)} \le \delta + \|f\|_{H^{1/2}(0,\infty)}\,.
\end{equation}
We use Fatou's lemma  first to deduce from~\eqref{pu30}  that $F\in H^{1/2}(0,\infty)$ with 
$$
\|F\|_{H^{1/2}(0,\infty)} \le \delta + \|f\|_{H^{1/2}(0,\infty)}\,,
$$ 
and then from \eqref{pu12} that
\begin{equation}
	\| f - F \|_{H^{1/2}(0,\infty)} \le \delta\,. \label{pu13}
\end{equation}
Similarly, for $k\ge 2$, it follows from \eqref{pu9} and \eqref{pu11} that
\begin{align}
	\|f-F\|_{L_2(I_k,\rd x/x)} & = \left\| \sum_{j=1}^{k+2} \varrho_{\varepsilon_j} * (\psi_j f) - \sum_{j=1}^{k+2} \psi_j f \right\|_{L_2(I_k,\rd x/x)} \nonumber \\
	& \le \sum_{j=1}^{k+2} \| \varrho_{\varepsilon_j} * (\psi_j f) - \psi_j f \|_{L_2(I_k,\rd x/x)} \le  \delta\,. \label{pu14}
\end{align}
In particular,
\begin{equation}
	\| F \|_{L_2(I_k,\rd x/x)} \le \delta + \|f\|_{L_2(I_k,\rd x/x)}\,, \label{pu15}
\end{equation}
and we invoke again Fatou's lemma to derive first from \eqref{pu15} that $F\in L_2((0,\infty),\rd x/x)$, and then from \eqref{pu14} that
\begin{equation}
	\|f - F \|_{L_2((0,\infty),\rd x/x)} \le \delta\,. \label{pu16}
\end{equation}
According to \eqref{pu13} and \eqref{pu16}, we have constructed a function $F\in C^\infty([0,\infty))\cap X$ lying in a $\delta$-neighborhood of $f$ in $X$. This result being valid whatever the value of $\delta\in (0,1)$, we have established the density of $C^\infty([0,\infty))\cap X$ in $X$. Finally, we use a standard truncation argument to deduce  that $C_c^\infty([0,\infty))\cap X$ is dense in $X$.
 
\smallskip
	
\noindent\textit{Step~2: Density of $C_c^\infty((0,\infty))$ in $X$.} We argue as in the proofs of \cite[Theorems~2.9.2~(c) \&~2.9.3~(d)]{Tr1978}. We fix $\chi\in C^\infty(\mathbb{R})$ such that $\chi(x)=0$ for $x\in (-\infty,1]$, $\chi(x)=1$ for $x\in [2,\infty)$, and $\chi(x)\in [0,1]$ for $x\in [1,2]$, and set $\chi_\lambda(x) := \chi(x/\lambda)$ for $x\in\mathbb{R}$ and $\lambda\in (0,1)$. 

Let $f\in X$ and $\delta\in (0,1)$. According to the previous step, there is $F\in C_c^\infty([0,\infty))\cap X$ such that
\begin{equation}
	\|f - F \|_{H^{1/2}(0,\infty)} + \|f - F \|_{L_2((0,\infty),\rd x/x)} \le \frac{\delta}{2}\,. \label{pu17}
\end{equation}
Since $F\in C_c^\infty([0,\infty))\cap X$, we observe that, for $\varepsilon\in (0,1)$, 
\begin{align*}
	\|F\|_{L_2((0,\infty),\rd x/x)} & \ge \int_0^1 |F(x)|^2\ \frac{\rd x}{x+\varepsilon} = \int_0^1 \left| F(0) + \int_0^x F'(y)\ \rd y \right|^2 \frac{\rd x}{x+\varepsilon} \\
	& \ge \int_0^1 \left[ \frac{|F(0)|^2}{2} - \left( \int_0^x F'(y)\ \rd y \right)^2 \right] \frac{\rd x}{x+\varepsilon} \\
	& \ge \frac{|F(0)|^2}{2} \ln\left( 1 + \frac{1}{\varepsilon} \right) - \|F'\|_{L_\infty(0,1)}^2 \int_0^1 \frac{x^2}{x+\varepsilon}\ \rd x \\
	& \ge \frac{|F(0)|^2}{2} \ln\left( 1 + \frac{1}{\varepsilon} \right) - \frac{\|F'\|_{L_\infty(0,1)}^2}{2}\,.
\end{align*}
Letting  $\varepsilon\to 0$ in the above inequality implies that $F(0)=0$, from which we deduce that 
\begin{equation}
	|F(x)|\le A_1 x \,, \qquad x\in [0,2]\,, \label{pu21}
\end{equation}
with $A_1 := \|F'\|_{L_\infty(0,2)}$. 

Now, for $\lambda\in (0,1)$, it follows from \eqref{pu21} that 
\begin{align}
	\| F - \chi_\lambda F\|_{L_2(0,\infty)} & = \left( \int_0^{2\lambda} \big( 1 - \chi_\lambda(x) \big)^2 |F(x)|^2\ \rd x \right)^{1/2} \nonumber \\
	& \le A_1 \left( \int_0^{2\lambda} x^2\ \rd x \right)^{1/2} \le \sqrt{3} A_1 \lambda^{3/2} \label{pu18}
\end{align}
and
\begin{align}
	\|F' - (\chi_\lambda F)'\|_{L_2(0,\infty)} & = \|(1-\chi_\lambda) F' + F \chi_\lambda'\|_{L_2(0,\infty)} \nonumber \\
	& \le \|(1-\chi_\lambda) F' \|_{L_2(0,\infty)} + \| F \chi_\lambda'\|_{L_2(0,\infty)} \nonumber \\
	& \le \left( \int_0^{2\lambda} |F'(x)|^2\ \rd x \right)^{1/2} \nonumber \\
	& \qquad + \frac{A_1}{\lambda} \left( \int_\lambda^{2\lambda} x^2 \left| \chi'\left( \frac{x}{\lambda} \right) \right|^2\ \rd x \right)^{1/2} \nonumber\\
	& \le \sqrt{2\lambda} \|F'\|_{L_\infty(0,2)} + \sqrt{\frac{7\lambda}{3}} A_1 \|\chi'\|_{L_\infty(\mathbb{R})} \nonumber\\
	& \le A_2 \sqrt{\lambda} \label{pu19}\,,
\end{align}
where $A_2 := 2 A_1(1+ \|\chi'\|_{L_\infty(\mathbb{R})})$. By interpolation, we deduce from \eqref{pu18} and \eqref{pu19} that
\begin{equation}
	\|F - \chi_\lambda F\|_{H^{1/2}(0,\infty)}  \le A \|F - \chi_\lambda F\|_{H^1(0,\infty)}^{1/2} \|F - \chi_\lambda F\|_{L_2(0,\infty)}^{1/2} \le A \lambda \label{pu20}
\end{equation}
for some constant $A>0$ depending only on $F$ and $\chi$. Similarly, by \eqref{pu21},
\begin{align}
	\|F - \chi_\lambda F \|_{L_2((0,\infty),\rd x/x)} & = \left( \int_0^{2\lambda} \big( 1-\chi_\lambda(x) \big)^2 |F(x)|^2 \frac{\rd x}{x} \right)^{1/2} \nonumber \\
	& \le A_1 \left( \int_0^{2\lambda} x\ \rd x \right)^{1/2} = \sqrt{2} A_1 \lambda\,. \label{pu22}
\end{align}
Thanks to \eqref{pu20} and \eqref{pu22}, there is $\lambda_\delta\in (0,1)$ such that 
\begin{equation*}
		\|F - \chi_{\lambda_\delta} F \|_{H^{1/2}(0,\infty)} + \|F - \chi_{\lambda_\delta} F \|_{L_2((0,\infty),\rd x/x)} \le \frac{\delta}{2}\,. 
\end{equation*}
Together with \eqref{pu17}, the above estimate ensures that 
\begin{equation*}
	\|f - \chi_{\lambda_\delta} F \|_{H^{1/2}(0,\infty)} + \|f - \chi_{\lambda_\delta} F \|_{L_2((0,\infty),\rd x/x)} \le \delta\,, 
\end{equation*}
and completes the proof, since $\chi_{\lambda_\delta} F\in C_c^\infty((0,\infty))$.
\end{proof}

\begin{proof}[Proof of Lemma~\ref{lem.d}]
The derivation of Lemma~\ref{lem.d} from Lemma~\ref{lem.d1} is also adapted from the proofs of \cite[Theorems~2.9.2~(c) \&~2.9.3~(d)]{Tr1978}. First, arguing as in the proof of \cite[Theorem~8.6]{Br2011}, we construct an extension operator 
$$
E\in \mathcal{L}\big(L_2(0,1),L_2(0,\infty)\big)\cap \mathcal{L}\big(H^{1}(0,1),H^1(0,\infty)\big)\,,
$$
$$
E\in  \mathcal{L}\big(L_2((0,1),\rd x/x),L_2((0,\infty),\rd x/x)\big)\,,
$$
and satisfies $Ef=f$ a.e. on $(0,1)$. By interpolation, $E\in\mathcal{L}\big(H^{1/2}(0,1),H^{1/2}(0,\infty)\big)$.

Now, let $f\in H^{1/2}(0,1)\cap L_2((0,1),\rd x/x)$ and $\delta\in (0,1)$.  Since $Ef$ belongs to $H^{1/2}(0,\infty)\cap L_2((0,\infty),\rd x/x)$, we infer from Lemma~\ref{lem.d1} that there is $\xi\in C_c^\infty((0,\infty))$ such that
\begin{equation}
	\| Ef - \xi \|_{H^{1/2}(0,\infty)} + \| Ef - \xi \|_{L_2((0,\infty),\rd x/x)} \le \frac{\delta}{2}\,. \label{x100}
\end{equation}	
We again fix $\chi\in C^\infty(\mathbb{R})$ such that $\chi(x)=0$ for $x\in (-\infty,1]$, $\chi(x)=1$ for $x\in [2,\infty)$, and $\chi(x)\in [0,1]$ for $x\in [1,2]$. For $\lambda\in (0,1/4)$ and $x\in (0,\infty)$, we set 
\begin{equation*}
	\xi_\lambda(x) := \xi(x) \chi\left( \frac{1-x}{\lambda} \right)\,,
\end{equation*}
and observe that 
\begin{equation*}
	\xi_\lambda(x) = 0 \;\text{ for }\; x\in [1-\lambda,\infty) \;\;\text{ and }\;\; \xi_\lambda(x) =\xi(x) \;\text{ for }\; x\in (0,1-2\lambda]\,.
\end{equation*}
Since $\xi\in C_c^\infty((0,\infty))$ and $\chi\in C^\infty(\mathbb{R})$, we conclude that $\xi_\lambda\in C_c^\infty((0,1))$ for all $\lambda\in (0,1/4)$. Furthermore,
\begin{equation}
	\begin{split}
	\|\xi - \xi_\lambda\|_{L_2(0,1)} & \le \sqrt{2\lambda} \|\xi\|_{L_\infty(0,1)} \,, \\
	\|\xi - \xi_\lambda\|_{L_2((0,1),\rd x/x)} & \le 2 \sqrt{\lambda} \|\xi\|_{L_\infty(0,1)} \,, \\
	\|\xi' - \xi_\lambda'\|_{L_2(0,1)} & \le \frac{2}{\sqrt{\lambda}} \|\xi\|_{C^1([0,1])} \,.	
	\end{split} \label{x103}
\end{equation}
By interpolation, we infer from \eqref{x103} that 
\begin{equation*}
	\|\xi - \xi_\lambda\|_{H^{1/2}(0,1)} \le b \|\xi - \xi_\lambda\|_{H^1(0,1)}^{1/2} \|\xi - \xi_\lambda\|_{L_2(0,1)}^{1/2} \le b \|\xi\|_{C^1([0,1])}
\end{equation*}
for some positive constant $b$ depending only on $\chi$. The above estimate ensures that $(\xi_\lambda)_{\lambda\in (0,1/4)}$ is bounded in $H^{1/2}(0,1)$, while it follows from \eqref{x103} that $(\xi_\lambda)_{\lambda\in (0,1/4)}$ converges to $\xi$ in $L_2((0,1),(1+1/x)\rd x)$. Consequently, there is a sequence $(\lambda_j)_{j\ge 1}$ in $(0,1/4)$, $\lambda_j\to 0$ as $j\to\infty$, such that
\begin{align}
	& \lim_{j\to\infty} \int_0^1 \big| \xi_{\lambda_j}(x) - \xi(x)  \big|^2 \left( 1+ \frac{1}{x} \right)\ \rd x = 0 \,, \label{x104a}\\
	& \xi_{\lambda_j} \rightharpoonup \xi \;\;\text{ in }\;\; H^{1/2}(0,1)\,. \label{x104b} 
\end{align}
According to \eqref{x104b} and Mazur's lemma, there is a sequence $(\bar{\xi}_k)_{k\ge 1}$ made up of convex combinations of $(\xi_{\lambda_j})_{j\ge 1}$ such that $(\bar{\xi}_k)_{k\ge 1}$ converges strongly to $\xi$ in $H^{1/2}(0,1)$. Each $\bar{\xi}_k$ being a convex combination of $(\xi_{\lambda_j})_{j\ge 1}$, it is obvious that $\bar{\xi}_k\in C_c^\infty((0,1))$ for each $k\ge 1$ and that \eqref{x104a} entails that
\begin{equation*}
	\lim_{k\to\infty} \int_0^1 \big| \bar{\xi}_k(x) - \xi(x)  \big|^2 \left( 1+ \frac{1}{x} \right)\ \rd x = 0 \,.
\end{equation*}
Therefore, there is $k_\delta\ge 1$ such that
\begin{equation}
	\big\|\bar{\xi}_{k_\delta} - \xi \big\|_{H^{1/2}(0,1)} + \big\|\bar{\xi}_{k_\delta} - \xi \big\|_{L_2((0,1),\rd x/x)} \le \frac{\delta}{2}\,. \label{x102}
\end{equation}
Recalling that $Ef=f$ a.e. in $(0,1)$, we combine \eqref{x100} and \eqref{x102} to conclude that $\bar{\xi}_{k_\delta}$ lies in a $\delta$-neighborhood of $f$ in $H^{1/2}(0,1)\cap L_2((0,1),\rd x/x)$, thereby completing the proof.
\end{proof}

%%%%%%%%%%%%%%%%
%%%%%%%%%%%%%%%%
\section{Proof of Proposition~\ref{C3}}\label{sec.pp3.3}
%%%%%%%%%%%%%%%%
%%%%%%%%%%%%%%%%

For $M>0$, we set $\Omega_M := D \times (-H,M)$ and define, for $v\in \bar{\mathcal{S}}$,
$$
G(v)[\theta]:=\left\{
\begin{array}{ll} \dfrac{1}{2}\displaystyle\int_{ \Omega(v)} \sigma \vert\nabla (\theta+h_v)\vert^2\,\rd (x,z)\,, & \theta\in H_0^1(\Omega(v))\,,\\
	\infty\,, & \theta\in L_2(\Omega_M)\setminus H_0^1(\Omega(v))\,.
\end{array}\right.
$$
We point out that $G(v)[\theta] = \mathcal{J}(v)[\theta+h_v]$ for $\theta\in H_0^1(\Omega(v))$. The next result is devoted to the stability of $G(v)$ with respect to $v$ which we express in terms of $\Gamma$-convergence of functionals.

%%%%%%%%%%%%%%%%
\begin{lemma}\label{P3}
	Consider $v\in \bar{\mathcal{S}}$ and a sequence $(v_n)_{n\ge 1}$ in $\bar{\mathcal{S}}$ satisfying \eqref{o1}. Then
	$$
	\Gamma-\lim_{n\rightarrow\infty} G(v_n)=G(v)\quad\text{in }\ L_2(\Omega_M)\,,
	$$
	where $M$ is defined in \eqref{z10} and $\Omega_M=D\times (-H,M)$.
\end{lemma}
%%%%%%%%%%%%%%%%

\begin{proof} The proof follows the lines of that of \cite[Proposition~3.11]{ARMA20}.  A first consequence of \eqref{o1} and the continuous embedding of $H_0^1(D)$ in $C(\bar D)$ is the uniform convergence
	\begin{equation}\label{o1b}
		v_n\rightarrow v\quad\text{in }\ C(\bar D)\,.
	\end{equation}
	
	\smallskip
	
	\noindent\textit{Step~1: Asymptotic Lower Semicontinuity.} Let $(\theta_n)_{n\ge 1}$ be an arbitrary sequence in $L_2(\Omega_M)$ and $\theta\in L_2(\Omega_M)$ such that
	\begin{equation}\label{t0}
		\theta_n\rightarrow \theta \ \text{ in }\ L_2(\Omega_M)\,.
	\end{equation}
	In order to prove that
	\begin{equation}\label{G1}
		G(v)[\theta]\le\liminf_{n\rightarrow\infty} G(v_n)[\theta_n]\,,
	\end{equation}
	we may assume without loss of generality that $\theta_n\in H_0^1(\Omega(v_n))$ for all $n\ge 1$ and that $(G(v_n)[\theta_n])_{n\ge 1}$ is bounded, since \eqref{G1} is clearly satisfied otherwise owing to the definition of $G$. Therefore, denoting the extension of $\theta_n$ by zero in $\Omega_M\setminus\Omega(v_n)$ by $\tilde\theta_n$, it follows from assumption~\eqref{203} that $(\tilde\theta_n)_{n\ge 1}$ is bounded in $H_0^1(\Omega_M)$, so that
	\begin{equation}\label{JJCale}
		(\tilde\theta_n)_{n\ge 1}\ \text{ is weakly relatively compact in }\ H_0^1(\Omega_M)\,.
	\end{equation}
	Introducing $\tilde\theta:=\theta\mathbf{1}_{\Omega(v)}$, we compute
	\begin{align*}
		\left\|\tilde{\theta}_n - \tilde{\theta}\right\|_{L_2(\Omega_M)}^2 & = \int_{\Omega(v_n)\cap \Omega(v)} |\theta_n-\theta|^2\, \rd (x,z) + \int_{\Omega(v_n)\cap(\Omega_M\setminus\Omega(v))} |\theta_n|^2\, \rd (x,z)\\
		& \qquad +\int_{(\Omega_M\setminus\Omega(v_n))\cap\Omega(v)} |\theta|^2\,\rd (x,z)\\
		& \le   \int_{\Omega(v_n)\cap \Omega(v)} |\theta_n-\theta|^2\,\rd (x,z) + 2 \int_{\Omega(v_n)\cap (\Omega_M\setminus\Omega(v))} |\theta_n-\theta|^2\, \rd (x,z)\\
		&  \qquad + 2 \int_{\Omega(v_n)\cap(\Omega_M\setminus\Omega(v))} |\theta|^2\, \rd (x,z) + \int_{(\Omega_M\setminus\Omega(v_n))\cap\Omega(v)} |\theta|^2\, \rd (x,z) \\
		& \le 2 \left\| \theta_n-\theta\right\|_{L_2(\Omega_M)}^2 + 2 \int_{\Omega(v_n)\cap(\Omega_M\setminus\Omega(v))} |\theta|^2\, \rd (x,z) \\
		& \qquad + \int_{(\Omega_M\setminus\Omega(v_n))\cap\Omega(v)} |\theta|^2\, \rd (x,z)\,.
	\end{align*}
	Owing to \eqref{o1b}, 
	\begin{equation*}
		\lim_{n\to \infty} \left| \Omega(v_n)\cap(\Omega_M\setminus\Omega(v)) \right| = \lim_{n\to\infty} \left| (\Omega_M\setminus\Omega(v_n))\cap\Omega(v) \right| = 0\,,
	\end{equation*}
	a property which, together with \eqref{t0} and Lebesgue's theorem, entails that the right-hand side of the above inequality converges to zero as $n\rightarrow \infty$. Consequently, $(\tilde\theta_n)_{n\ge 1}$ converges to $\tilde\theta$ in $L_2(\Omega_M)$.  This property, along with \eqref{JJCale}, implies that $\tilde\theta\in H_0^1(\Omega_M)$ and, bearing in mind that $\Omega(v)\subset \Omega_M$ and the compactness of the embedding of $H^1(\Omega(v))$ in $H^{3/4}(\Omega(v))$,
	\begin{equation}\label{hh}
		\tilde\theta_n \rightharpoonup \tilde\theta \ \text{ in }\ H^1(\Omega_M)\,,\qquad \tilde\theta_n \rightarrow \theta \ \text{ in }\ H^{3/4}(\Omega(v))\,.
	\end{equation}
	Invoking \eqref{204} and the continuity of the trace, we deduce
	\begin{equation}\label{o2}
		\tilde\theta_n+h_{v_n}\rightarrow \theta +h_{v}\quad\text{ in }\ L_2(\partial\Omega(v))\,.
	\end{equation}
	We now claim that \eqref{hh} and \eqref{o2} imply $\theta\in H_0^1(\Omega(v))$. Indeed, on the one hand, $\theta=\tilde\theta$ vanishes on $D\times \{-H\}$ and on $\{\pm L\}\times (-H,d)$. On the other hand, we infer from H\"older's inequality that
	\begin{align*}
		&\left| h_{v_n}(x,v_n(x)+d)-(\tilde\theta_n+h_{v_n})(x,v(x)+d)\right|\\
		&\qquad\qquad= \left| (\tilde\theta_n+h_{v_n})(x,v_n(x)+d) - (\tilde\theta_n+h_{v_n})(x,v(x)+d) \right|\\
		&\qquad\qquad = \left| \int_{v(x)+d}^{v_n(x)+d}\partial_z (\tilde\theta_n+h_{v_n})(x,z)\, \rd z\right|\\
		&\qquad\qquad\le | v_n(x)-v(x)|^{1/2} \left(\int_{-H}^M |\partial_z (\tilde\theta_n+ h_{v_n}) (x,z)|^2\, \rd z\right)^{1/2}
	\end{align*}
	for a.e. $x\in D$, and thus
	\begin{align*}
		& \int_D |h_{v_n}(x,v_n(x)+d)-(\tilde\theta_n+h_{v_n})(x,v(x)+d)|^2\,\rd x\\
		& \qquad\qquad \le \int_D |v_n(x)-v(x)| \int_{-H}^{M} |\partial_z (\tilde\theta_n+h_{v_n})(x,z)|^2\, \rd z\rd x\\
		&\qquad\qquad \le \frac{\|v_n-v\|_{L_\infty(D)}}{\min\{\sigma_1,\sigma_2\}} \int_{\Omega_M} \sigma |\nabla (\tilde\theta_n+h_{v_n})(x,z)|^2\, \rd (x,z)\\
		& \qquad\qquad = 2 \frac{\|v_n-v\|_{L_\infty(D)}}{\min\{\sigma_1,\sigma_2\}}\,  \left\{G(v_n)[\theta_n] +\int_{\Omega_M\setminus\Omega(v_n)}\sigma |\nabla h_{v_n}(x,z)|^2\, \rd (x,z) \right\}\,.
	\end{align*}
	Note that the sum embraced with curly brackets is bounded due to the boundedness of $(G(v_n)[\theta_n])_{n\ge 1}$ and~ \eqref{204}. Consequently, the uniform convergence ~\eqref{o1b} guarantees that the right-hand side of the above inequality converges to zero as $n\to\infty$. Hence, due to \eqref{205} and \eqref{o2}, we conclude that $\theta=0$ on $\{(x,v(x)+d)\,:\, x\in D\}$ and thus $\theta\in H_0^1(\Omega(v))$. Now, by~\eqref{204} and~\eqref{hh},
	$$
	\tilde\theta_n+h_{v_n} \rightharpoonup\tilde\theta+h_{v}\quad\text{in}\quad H_0^1(\Omega_M) \,,
	$$
	so that
	\begin{equation}\label{p1}
		\int_{\Omega_M}\sigma \vert \nabla(\tilde\theta+h_v)\vert^2\,\rd (x,z) \le
		\liminf_{n\rightarrow\infty} \int_{\Omega_M}\sigma \vert \nabla(\tilde\theta_n+h_{v_n})\vert^2\,\rd (x,z)\,.
	\end{equation}
	Since $\tilde\theta_n\in H_0^1(\Omega(v_n))$, 
	$$
	\int_{\Omega_M\setminus\Omega(v_n)}\sigma \vert \nabla(\tilde\theta_n+h_{v_n})\vert^2\,\rd (x,z)=\int_{\Omega_M\setminus\Omega(v_n)}\sigma \vert \nabla h_{v_n}\vert^2\,\rd (x,z)\,,
	$$
	and we deduce from~\eqref{204} that
	\begin{equation}
		\begin{split}\label{p2}
			\lim_{n\rightarrow\infty } \int_{\Omega_M\setminus\Omega(v_n)}\sigma \vert \nabla(\tilde\theta_n+h_{v_n})\vert^2\,\rd (x,z) & =
			\int_{\Omega_M\setminus\Omega(v)}\sigma \vert \nabla h_{v}\vert^2\,\rd (x,z)\\
			& = \int_{\Omega_M\setminus\Omega(v)}\sigma \vert \nabla (\tilde\theta+h_{v})\vert^2\,\rd (x,z)\,,
		\end{split}
	\end{equation}
	the last equality being due to $\tilde\theta\in H_0^1(\Omega(v))$.	Combining  \eqref{p1} and \eqref{p2} gives \eqref{G1}.
	
	\smallskip
	
	\noindent\textit{Step~2: Existence of a Recovery Sequence.} We only need to provide a recovery sequence for $\theta\in H_0^1(\Omega(v))$, in view of the definition of the functional $G(v)$. Note that $\theta\in H_0^1(\Omega_M)$ and that $f:=-\Delta \theta\in H^{-1}(\Omega_M)$ can be considered also as an element of $H^{-1}(\Omega(v_n))$ by restriction. Let now $\theta_n\in H_0^1(\Omega(v_n))$ denote the unique weak solution to
	$$
	-\Delta \theta_n=f \quad\text{in }\ \Omega(v_n)\,,\qquad \theta_n=0 \quad\text{on }\ \partial\Omega(v_n)\,.
	$$
	Since the Hausdorff distance $d_H$ in $\Omega_M$ (see \cite[Section~2.2.3]{HP05}) satisfies
	$$
	d_H(\Omega(v_n),\Omega(v)) \le \|v_n - v \|_{L_\infty(D)}\rightarrow 0
	$$
	by \eqref{o1b} and since $\overline{\Omega_M}\setminus\Omega(v_n)$ has a single connected component for every $n\ge 1$ as $v_n\ge -H$, it follows from \cite[Theorem~4.1]{Sv93} and \cite[Theorem~3.2.5]{HP05} that 
	$\theta_n\rightarrow \hat\theta$ in $H_0^1(\Omega_M)$,
	where $\hat\theta\in H_0^1(\Omega_M)$ is the unique weak solution to
	\begin{equation*}
		-\Delta \hat\theta=f =-\Delta\theta \quad\text{in }\ \Omega_M\,,\qquad \hat\theta=0 \quad\text{on }\ \partial\Omega_M\,. \label{z11}
	\end{equation*}
	Since $\hat\theta$ and $\theta$ both belong to $H_0^1(\Omega_M)$, we readily deduce from the above equation that $\hat\theta=\theta$, so that $\theta_n\rightarrow \theta$ in $H_0^1(\Omega_M)$. Since $\theta_n\in H_0^1(\Omega(v_n))$ and $\theta\in H_0^1(\Omega(v))$, this convergence yields, along with \eqref{204} and \eqref{203},
	\begin{align*}
		\int_{\Omega(v)}\sigma | \nabla(\theta+h_v)|^2\,\rd (x,z)
		& = \int_{\Omega(v)}\sigma \left( |\nabla\theta|^2 + 2 \nabla\theta\cdot \nabla h_v + |\nabla h_v|^2\right)\,\rd (x,z)\\
		& = \lim_{n\rightarrow\infty} \int_{\Omega_M}\sigma \left( |\nabla\theta_n|^2 + 2 \nabla\theta_n\cdot \nabla h_{v_n}\right)\,\rd (x,z)\\
		&\qquad + \lim_{n\rightarrow\infty} \int_{\Omega(v_n)}\sigma |\nabla h_{v_n}|^2\,\rd (x,z)\\
		&= \lim_{n\rightarrow\infty}\int_{\Omega(v_n)}\sigma |\nabla(\theta_n +h_{v_n})|^2\,\rd (x,z)\,;
	\end{align*}
	that is,
	$$
	G(v)[\theta]=\lim_{n\rightarrow\infty} G(v_n)[\theta_n]\,.
	$$
	Hence, $(\theta_n)_{n\ge 1}$ is a recovery sequence for $\theta$. 
	
	\smallskip 
	
	Thanks to the just established two properties, we have proved the $\Gamma$-convergence of $(G(v_n))_{n\ge 1}$ to $G(v)$ in $L_2(\Omega_M)$.
\end{proof}

Proposition~\ref{C3} is now an almost immediate consequence of Lemma~\ref{P3}.

\begin{proof}[Proof of Proposition~\ref{C3}]  For $n\ge 1$, set 
	$$
	\chi_n:=\psi_{v_n}-h_{v_n}\in H_0^1(\Omega(v_n))\subset H_0^1(\Omega_M)\,,
	$$ 
	and recall that $\chi_n$ is a minimizer of $G(v_n)$ in $H_0^1(\Omega(v_n))$ by Lemma~\ref{L1} and $\Omega_M = D\times (-H,M)$. Since  $(v_n)_{n\ge 1}$ is bounded in $H^1(D)$, it follows from Lemma~\ref{L1}, Lemma~\ref{LP}, and~\eqref{203} that $(\chi_n)_{n\ge 1}$ is bounded in $H_0^1(\Omega_M)$. Hence, there are a subsequence $(n_j)_{j\ge 1}$ and $\chi\in H_0^1(\Omega_M)$ such that $\chi_{n_j}\rightarrow \chi$ in $L_2(\Omega_M)$ and $\chi_{n_j}\rightharpoonup \chi$ in $H_0^1(\Omega_M)$. By Lemma~\ref{P3} and the Fundamental Theorem of $\Gamma$-Convergence, see \cite[Corollary~7.20]{DaM93}, $\chi$ is a minimizer of the functional $G(v)$ on $L_2(\Omega_M)$. Clearly, from the definition of $G(v)$, we see that $\chi+h_v\in \mathcal{A}(v)$ minimizes the functional $\mathcal{J}(v)$ on $\mathcal{A}(v)$, hence $\psi_v=\chi+h_v$ by Lemma~\ref{L1}. Consequently, $\chi=\psi_v-h_v$ is the unique cluster point of the sequence $(\chi_n)_{n\ge 1}$ in $L_2(\Omega_M)$ and this sequence is compact in that space and weakly compact in $H_0^1(\Omega_M)$. Combining these properties leads us to conclude that $\chi_{n}\rightarrow \chi$ in $L_2(\Omega_M)$ and $\chi_{n}\rightharpoonup \chi$ in $H_0^1(\Omega_M)$. Moreover, the Fundamental Theorem of $\Gamma$-Convergence, see  \cite[Corollary~7.20]{DaM93}, also ensures that
	$$
	\lim_{n\to\infty} \mathcal{J}(v_n)[\psi_{v_n}] = \lim_{n\to\infty} G(v_n)[\chi_n] = G(v)[\chi] =\mathcal{J}(v)[\psi_v]\,.
	$$ 
	In particular, this property, along with~\eqref{204}, implies that 
	$$
	\lim_{n\rightarrow\infty} \|\chi_n\|_{H_0^1(\Omega_M)}= \|\chi\|_{H_0^1(\Omega_M)}\,.
	$$
	Since $(\chi_n)_{n\ge 1}$ converges weakly to $\chi$ in $H_0^1(\Omega_M)$, this gives the strong convergence of $(\chi_n)_{n\ge 1}$ in $H_0^1(\Omega_M)$.
\end{proof}

%%%%%%%%%%%%%%%%
%%%%%%%%%%%%%%%%
\section{Proof of Proposition~\ref{ACDC}}\label{pr.acdc}
%%%%%%%%%%%%%%%%
%%%%%%%%%%%%%%%%

We first establish Lemma~\ref{lez1} which is the building block of the proof of Proposition~\ref{ACDC}.

\begin{proof}[Proof of Lemma~\ref{lez1}] It first follows from \eqref{vh2} and the continuous embedding of $H^2(D)$ in $C^1(\bar{D})$ that
	\begin{equation*}
		M = d + \max\left\{ \|v\|_{L_\infty(D)} \,,\, \sup_{n\ge 1}\{\|v_n\|_{L_\infty(D)}\} \right\} < \infty 
	\end{equation*}
	and 
	\begin{equation}
		\lim_{n\to\infty} \|v_n-v\|_{C^1(\bar{D})} = 0\,. \label{z100}
	\end{equation}
	We may then apply Proposition~\ref{C3} to obtain that
	\begin{equation}
		\psi_{v_n}-h_{v_n}\rightarrow  \psi_v-h_v\quad\text{in }\ H_0^1(D \times (-H,M))\,. \label{queen}
	\end{equation}
	
	Now, fix $i\in\{1,2\}$ and let  $U_i$ be any open subset of $\Omega_i(v)$ such that $\bar{U}_i$ is a compact subset of $\Omega_i(v)$. Owing to \eqref{z100}, there is an integer $N\ge 1$ such that $U_i\subset \Omega_i(v_n)$ for $n\ge N$. We then infer from \eqref{204}, \eqref{uh2}, and \eqref{queen} that $(\psi_{v_n,i})_{n\ge N} $ is bounded in $H^2(U_i)$ and $\psi_{v_n,i}\to \psi_{v,i}$ in $H^1(U_i)$ as $n\to\infty$. Thus, $(\psi_{v_n,i})_{n\ge N}$ is a weakly compact sequence in $H^2(U_i)$ which has a unique cluster point $\psi_{v,i}$ for that topology, so that $\psi_{v,i}\in H^2(U_i)$ and \begin{equation}
		\psi_{v_n,i}\rightharpoonup \psi_{v,i} \;\;\text{ in }\;\; H^2(U_i) \;\;{ as }\;\; n\to\infty\,. \label{100}
	\end{equation} 
	In particular, we deduce from \eqref{uh2} and \eqref{100} that
	\begin{align*}
		\int_{U_i} |\partial_x^l\partial_z^k \psi_{v,i} |^2\,\rd (x,z) & \le \liminf_{n\rightarrow \infty} \int_{U_i} |\partial_x^l\partial_z^k \psi_{v_n,i} |^2\,\rd (x,z) \\
		& \le \sup_{n\ge 1}\left\{ \int_{\Omega_i(v_n)} |\partial_x^l\partial_z^k \psi_{v_n,i}\vert^2\,\rd (x,z) \right\}
	\end{align*}
	for $(l,k)\in \{(2,0),(1,1),(0,2)\}$. We then use Fatou's lemma to conclude that $\partial_x^l\partial_z^k \psi_{v,i}$ belongs to $L_2(\Omega_i(v))$ for $(l,k)\in \{(2,0),(1,1),(0,2)\}$ with
	\begin{equation*}
		\int_{\Omega_i(v)} |\partial_x^l\partial_z^k \psi_{v,i} |^2\,\rd (x,z) \le \sup_{n\ge 1}\left\{ \int_{\Omega_i(v_n)} |\partial_x^l\partial_z^k \psi_{v_n,i}\vert^2\,\rd (x,z) \right\}\,.
	\end{equation*}
	Therefore, $\psi_{v,i}\in H^2(\Omega_i(v))$ and \eqref{z101} readily follows from \eqref{uh2} and the above estimate.
\end{proof}

\begin{proof}[Proof of Proposition~\ref{ACDC}~\textbf{(a)}]
	Let $v\in \bar{\mathcal{S}}$ be such that $\|v\|_{H^2(D)}\le \kappa$. We may choose a sequence $(v_n)_{n\ge 1}$ in $\mathcal{S}\cap W_\infty^2(D)$ satisfying 
	\begin{equation}
		v_n\rightarrow v \ \text{ in }\ H^2(D)\,,\qquad \sup_{n\ge 1}\,\|v_n\|_{H^2(D)}\le 2\kappa\,. \label{z103}
	\end{equation}
	Owing to \eqref{z103} and the regularity property $v_n\in \mathcal{S}\cap W_\infty^2(D)$, $n\ge 1$, Proposition~\ref{C3} guarantees that $(\psi_{v_n,1},\psi_{v_n,2})$ belongs to $H^2(\Omega_1(v_n))\times H^2(\Omega_2(v_n))$ and $(\psi_{v_n})_{n\ge 1}$ satisfies \eqref{uh2} with $\mu=c_0(2\kappa)$. We then infer from Lemma~\ref{lez1} that $(\psi_{v,1},\psi_{v,2})$ belongs to $H^2(\Omega_1(v))\times H^2(\Omega_2(v))$ and satisfies
	\begin{equation*}
		\|\psi_{v,1}\|_{H^2(\Omega_1(v))} + \|\psi_{v,2}\|_{H^2(\Omega_2(v))} \le c_0(2\kappa)\,. 
	\end{equation*}
	Combining the above bound with \eqref{204} and Lemma~\ref{lez1} gives \eqref{king3}. 
	
We now check that $\psi_v$ is a strong solution to \eqref{psiS}. 	As a minimizer of $\mathcal{J}(v)$ on $\mathcal{A}(v)$ according to Lemma~\ref{L1}, the function $\psi_v$ satisfies
	\begin{equation}\label{e1}
		\int_{\Omega(v)}\sigma\nabla\psi_v\cdot\nabla \theta\,\rd (x,z)=0\,,\quad \theta\in H_0^1(\Omega(v))\,.
	\end{equation}
	Thus, since $(\psi_{v,1},\psi_{v,2})\in H^2(\Omega_1(v))\times H^2(\Omega_2(v))$ it readily follows that $\mathrm{div}(\sigma\nabla\psi_v)=0$ in $\Omega(v)$ as claimed in \eqref{a1aS}. Moreover, owing to $v\in C(\bar D)$, we can write the open set $\{x\in D\,:\, v(x)>-H\}$  as a countable union of open intervals $((a_i,b_i))_{i\in I}$, see \cite[IX.Proposition~1.8]{AEIII}. 
	Let $i\in I$ and set
	$$
	O_i(v):=\{(x,z)\in (a_i,b_i)\times\R\,:\, -H<z<v(x)+d\}\subset\Omega(v)\,.
	$$  For each $\theta\in \mathcal{D}(O_i(v))$ we infer from \eqref{e1} and Gau\ss' theorem that
	\begin{equation*}
		0=\int_{O_i(v)}\sigma\nabla\psi_v\cdot\nabla \theta\,\rd (x,z)  =\int_{a_i}^{b_i} \left(\llbracket \sigma\nabla \psi_v \rrbracket \cdot {\bf n}_{ \Sigma(v)}\theta\right) (x,v(x))\,\rd x\,,
	\end{equation*}
	hence $\llbracket \sigma\nabla \psi_v \rrbracket \cdot {\bf n}_{ \Sigma(v)} (\cdot,v(\cdot))=0$ a.e. in $(a_i,b_i)$. Therefore, $\llbracket \sigma\nabla \psi_v \rrbracket \cdot \mathbf{n}_{ \Sigma(v)} =0$ on $\Sigma(v)$ as stated in \eqref{a1bS}. Finally, since $\psi_v\in H^1(\Omega(v))$ we have $\llbracket \psi_v \rrbracket =0$ on $\Sigma(v)$, while \eqref{a1cS} is due to $\psi_v\in\mathcal{A}(v)$.
\end{proof}

\begin{proof}[Proof of Proposition~\ref{ACDC}~\textbf{(b)}]
	Proposition~\ref{ACDC}~\textbf{(b)} is now a straightforward consequence of Proposition~\ref{ACDC}~\textbf{(a)} and Lemma~\ref{lez1}.
\end{proof}

%%%%%%%%%%%%%%%%
%%%%%%%%%%%%%%%%
\bibliographystyle{siam}
\bibliography{MEMSBIB}
%%%%%%%%%%%%%%%%
%%%%%%%%%%%%%%%%

%%%%%%%%%%%%%%%%
%%%%%%%%%%%%%%%%
\end{document}